\documentclass[12pt]{amsart} % use larger type; default would be 10pt

\usepackage[utf8]{inputenc} % set input encoding (not needed with XeLaTeX)
\usepackage{bbm}

\usepackage{graphicx} % support the \includegraphics command and options

 \usepackage[parfill]{parskip} % Activate to begin paragraphs with an empty line rather than an indent
 
\usepackage{booktabs} % for much better looking tables
\usepackage{array} % for better arrays (eg matrices) in maths
\usepackage{paralist} % very flexible & customisable lists (eg. enumerate/itemize, etc.)
\usepackage{verbatim} % adds environment for commenting out blocks of text & for better verbatim
\usepackage{subfig} % make it possible to include more than one captioned figure/table in a single float
\usepackage{mathrsfs}
\usepackage{amssymb}
\usepackage{amsthm}
\usepackage{amsmath,amsfonts,amssymb,esint}
\usepackage{graphics,color}
\usepackage{enumerate}
\usepackage{mathtools,centernot}
\usepackage{cases}

\newtheorem*{remark 1}{Remark}

\newtheorem{lemma}{Lemma}
\newtheorem{theorem}{Theorem}
\newtheorem{corollary}{Corollary}
\newtheorem{proposition}{Proposition}
\newtheorem{remark}{Remark}

\newcommand{\abs}[1]{\left|#1\right|}
\DeclarePairedDelimiter{\floor}{\lfloor}{\rfloor}

\newcommand{\res}{\text{Res}}

\usepackage{ color, graphicx}
\usepackage{mathrsfs, dsfont}

\usepackage{hyperref}

\numberwithin{equation}{section}

\newcommand{\RR}{\mathbb{R}}

\newcommand{\R}{\mathbb{R}}

\newcommand{\Z}{\mathbb{Z}}

\newcommand\<{\langle}\renewcommand\>{\rangle}

\newcommand\nc{\newcommand}
\nc\hd{\widehat{D}}
\nc\be{\begin{equation}}
\nc\ee{\end{equation}}
\nc\kp{\kappa}
\nc\8{\infty}

\DeclareMathSymbol{\Gamma}{\mathalpha}{letters}{"00}
\DeclareMathSymbol{\Theta}{\mathalpha}{letters}{"02}
\DeclareMathSymbol{\Lambda}{\mathalpha}{letters}{"03}
\DeclareMathSymbol{\Omega}{\mathalpha}{letters}{"0A}
\definecolor{cr}{rgb}{1,0,0}

\newcommand{\supp}{\operatorname{supp}} 

\title{Effective dynamics of the nonlinear Schr\"odinger equation on large domains}
\author{T. Buckmaster, P. Germain, Z. Hani, J. Shatah}

\begin{document}

\begin{abstract}
We consider the nonlinear Schr\"odinger (NLS) equation posed on the box $[0,L]^d$ with periodic boundary conditions. The aim is to describe the long-time dynamics by deriving \emph{effective equations} for it when $L$ is large and the characteristic size $\epsilon$ of the data is small. Such questions arise naturally when studying dispersive equations that are posed on large domains (like water waves in the ocean), and also in theory of statistical physics of dispersive waves, that goes by the name of ``wave turbulence". Our main result is deriving a new equation, the \emph{continuous resonant} (CR) equation, that describes the effective dynamics for large $L$ and small $\epsilon$ over very large time-scales. Such time-scales are well beyond the (a) nonlinear time-scale of the equation, and (b) the Euclidean time-scale at which the effective dynamics are given by (NLS) on $\RR^d$. The proof relies heavily on tools from analytic number theory, such as  a relatively modern version of the  Hardy-Littlewood circle method, which are modified and extended to be applicable in a PDE setting.
\end{abstract}

\maketitle

%\tableofcontents

\section{Introduction}

\subsection{The nonlinear Schr\"odinger equation on the torus}  The first fundamental question, beyond well-posedness, in the study of nonlinear dispersive equations is the long time dynamics of small amplitude solutions. There is a large body of work when the domain is $\mathbb{R}^n$,  where for most equations, the asymptotic behavior of small solutions is well-understood \cite{TaoBook, CazenaveBook}. However, the situation is markedly different on bounded domains, where solutions, even for small initial data, exhibit rich and complicated  dynamical behaviors. These can range from quasi-periodic motion \cite{ProcesiProcesi} to solutions that exhibit a very vigorous departure from linear behavior, like energy cascades between different length-scales.

The  purpose of this manuscript is to study  the long time dynamics of the nonlinear Schr\"odinger equation
\begin{equation}\label{pNLS}
\left\{ \begin{array}{l}
- i \partial_t v + \frac{1}{2\pi} \Delta v =\pm |v|^{2p} v, \qquad   x \in  \mathbb{T}^n_L ,\;   p\in \mathbb{N} \\
v(t=0) = \epsilon v_0,
\end{array} \right.
\end{equation}
where  $\mathbb{T}^n_L$ is the torus of size $L$, i.e., the box $[0,L]^n$ with periodic boundary conditions.

Setting for the moment $p= \epsilon = L = 1$, local well-posedness for smooth data trivially holds.  However, Bourgain~\cite{bourgain_fourier_1993} in a foundational work,  obtained local well-posedness  for
data in  $H^{\frac{n}{2}-\frac{1}{p}+}$ (slightly smoother than the scale invariant space).  Combining this local result with the conserved quantities of the equation leads directly to global solutions for $n=1,2,3$ in the defocusing case, and for small data in the focusing case. Similar global results for were later obtained in the energy-critical case (see \cite{HerrTatTzv, Bourgain2013, IoPa}). This raises the question: what are the qualitative features of these global solutions? It is expected that for generic initial data, solutions will exhibit a behavior that is markedly different than the linear one (no matter how small the size of the data is!). An example of such behavior is the energy cascade phenomenon, in which solutions transfer their energy to higher and higher Fourier modes. This can be measured by inspecting the behavior of high Sobolev norms, in which case one expects to find an abundance of solutions which satisfy for $s>1$,  $\limsup_{t \to \infty} \|u(t)\|_{H^s} = \infty$. Although this has not been proven for any solution on the torus $\mathbb{T}^n$ (though see \cite{HPTV}), there are solutions which exhibit arbitrarily large growth factor~\cite{CKSTT,GK,Hani,HausProc, GuaHauProc}.

It should be mentioned that the intuition for energy cascades is highly motivated by the theory of statistical physics of nonlinear dispersive waves, namely \emph{``wave or weak turbulence"} \cite{ZLF,Nazarenko},\footnote{The name \emph{weak turbulence} stems from the focus of this theory on small solutions, and hence weak nonlinearities.}
%can be seen as a manifestation of the physical theory of weak turbulence~\cite{ZLF,Nazarenko}.  
Unfortunately, this theory is, so far, lacking mathematical foundations (see~\cite{LS} for the best result in that spirit). The most striking element of the theory of wave turbulence, is the derivation of a kinetic model which should describe large time dynamics. The derivation of this kinetic equation is performed under a randomness assumption, and in the limit where $\epsilon \to 0$, $L \to \infty$.  This was one motivation for the central aim of this article: describe the long time dynamics of $u$, on a time-scale $T\to \infty$,  as $\epsilon \rightarrow 0$ (weak nonlinearity) and as $L \rightarrow \infty$ (big box limit, or, up to rescaling, high frequency limit).

\subsection{Effective dynamics on large domains} Another motivation to study the long-time behavior in the large box limit, is to understand the various regimes and effective dynamics that are featured by dispersive systems on large domains (e.g. water waves equation posed on the ocean surface). Indeed, for such systems, one is often tempted, for modeling and mathematical reasons, to find simplified equations \emph{posed on $\RR^n$} that approximate the real dynamics when $L$ is very large. The first and most intuitive such simplification, is to study the same equation on $\RR^n$. We call this the \emph{Euclidean approximation}. However, one soon notices that this approximation has its limitations, namely it is only valid in situations and over time-scales for which the solution does not feel the effect of the boundary of the domain. 

To explain this more precisely, let us consider again our equation \eqref{pNLS}. We first identify two important time-scales for the dynamcis:

\underline{The nonlinear time-scale $T_{nl}$}: this is the time needed for the nonlinearity to have an effect of size $O(1)$. It is easy to see that for initial data of size $\epsilon$ and a nonlinearity of degree $2p+1$, the nonlinear time scale is $T_{nl}\sim \epsilon^{-2p}$.  Therefore keeping $L$ fixed and letting $\epsilon \to 0$,  we have: {\it 1)}  for  $t\ll T_{nl}$  the dynamics of $u$ is given (to leading order) by the linear Schr\"odinger equation on $\mathbb T^n_L$;  and  {\it 2)} for  $t= O(T_{nl})$  the dynamics is given by the nonlinear Schr\"odinger equation on $\mathbb T^n_L$.

\underline{The Euclidean time-scale $T_{\mathcal E}$}: this is the time needed for the solution to be affected by the geometry of the domain $\mathbb{T}^n_L$. Since at the linear level, wave packets at frequency scale~$\sim 1$ move at speed~$\sim 1$, then one can heuristically argue that a scale~$\sim 1$ wave packet, initially localized, would take time $O(L)$ to wrap up the torus; therefore $T_{\mathcal E}\sim L$.  Thus for $\epsilon$ fixed and $L$  large, the dynamics of $u$ is well approximated for $t\ll T_{\mathcal E}$ by the  nonlinear Schr\"odinger equation on $\mathbb R^n$.

%In the present paper we want to answer the following question:
This leads to the main question of the present work:

\smallskip 
{\it  What happens after the Euclidean approximation breaks, i.e., $T\gg L$ and after the nonlinear effects take place, i.e.,  $T \gg \epsilon^{-2p}$? Is there an equation that describes the  dynamics of solutions when $\epsilon$ is small and  $L$ is large?}

\subsection{The resonant time-scale and rough statement of results} The purpose of this paper is to contribute to the answer to the above open-ended question. Roughly speaking, our result can be stated as follows: There exists another time-scale $T_R$, that we call the \emph{resonant time-scale}, that is much longer than both $ T_{nl}$ and $T_{\mathcal E}$, and over which we can still describe the effective dynamics by an equation on $\RR^n$. We call this equation the \emph{Continuous Resonant} (CR) equation. The rigorous derivation of this equation, along with proving how it approximates the (NLS) dynamics in the limit of large $L$, constitutes the bulk of this paper. In a companion paper \cite{BGHS}, we will analyze this equation and study more of its properties.

Let us start by giving a formal derivation of this equation. First, for simplicity of the presentation, we set $p=1$  and choose the defocusing $+$ sign of the nonlinearity in \eqref{NLS} (both of which have little role in what follows). We start with an ansatz $v =\epsilon u$ to emphasize the size of the initial data under consideration. The equation satisfied by $u$ is given by 
\begin{equation}
\label{NLS}
\boxed{ \tag{NLS} 
-i \partial_t u + \frac{1}{2\pi} \Delta u = \epsilon^{2} |u|^2 u, \qquad \mbox{$x \in \mathbb{T}^n_L$}.
}
\end{equation}
Expand $u$ in its Fourier coefficients, $u(t,x) = \frac{1}{L^n} \sum_{K \in \mathbb{Z}^n_L} \widehat{u}_K(t) e(K \cdot x),
$
where $K \in \mathbb{Z}^n_L = \left( \frac{\mathbb{Z}}{L} \right)^n$ and $e(\alpha) = \exp(2\pi i \alpha)$, and  define $a_K(t) = e(-|K|^2 t) \widehat{u}_K(t)$. The equation satisfied by $a_K$ reads
\begin{equation}\label{eq:profile}
- i \partial_t a_K = \frac{\epsilon^2}{L^{2n}}\sum_{\mathcal{S}_3(K) =0}  
a_{K_1} \overline{a_{K_2}} a_{K_3} e(\Omega_3(K)t),
\end{equation}
where $K_i\in  \mathbb{Z}^n_L$, for $i=1, 2, 3$,  and
\begin{align*}
& \mathcal{S}_3(K) = K_1 - K_2 + K_3 - K,\\
& \Omega_3(K) = |K_1|^2 - |K_2|^2 + |K_3|^2 - |K|^2. 
\end{align*}
We  split the nonlinear terms into resonant and non-resonant interactions:
\begin{equation}
\label{resornonres}
- i \partial_t a_K = \underbrace{\frac{\epsilon^2}{L^{2n}}\sum_{\substack{\mathcal{S}_3(K) =0\\ \Omega_3(K)=0}} a_{K_1} \overline{a_{K_2}} a_{K_3}}_{\mbox{resonant interactions}} + \underbrace{\frac{\epsilon^2}{L^{2n}}\sum_{\substack{\mathcal{S}_3(K) =0 \\ \Omega_3(K)\neq 0}} a_{K_1} \overline{a_{K_2}} a_{K_3} e(\Omega_3(K)t)}_{\mbox{non-resonant interactions}},
\end{equation}
and prove the following:
%
%\comment{ZH: Maybe here we could say something in the direction of: The large box limit and the small nonlinearity limit are also interesting from a purely dynamical and modeling perspective. Indeed, many PDE model phenomona that are posed on very large domains (like water waves in the ocean), and it is often useful, for mathematical and physical reasons,  to make an infinite volume apprximation and work with an equation on $\R^d$ instead of an equation on a very large compact domain. But this begs the question: what is the correct infinite volume approximation?}

{\it 1)}  For $\epsilon$ sufficiently small, the non-resonant interactions become dynamically irrelevant, and   the dynamics of small solutions are  well-approximated by the resonant terms only. The proof of this requires the use of normal form transformations. Eliminating the cubic non-resonant terms by a normal forms transformation justifies this approximation under the restrictive condition $\epsilon< L^{-1-}$  (cf. \cite{FGH}).  In this paper, we eliminate non-resonant terms up to an arbitrary large (but finite) degree. This allows us to justify this resonant approximation under the mild condition $\epsilon< L^{-\gamma}$ for arbitrary small $\gamma>0$. This is done in Section \ref{sec:NF-section}.  The upshot is that effectively, the dynamics of $a_K(t)$ are given by
\[
- i \partial_t a_K = \frac{\epsilon^2}{L^{2n}}\mathcal{T}_L(a,a,a) \qquad \mbox{where  }\,\, \mathcal{T}_L(a,a,a) = \sum_{\substack{\mathcal{S}_3(K) =0\\ \Omega_3(K)=0}} a_{K_1} \overline{a_{K_2}} a_{K_3}~.
\]

{\it 2)}  For $L$ sufficiently large, one can approximate the resonant sum above by an integral in a manner  similar to how Riemann sums are approximated by integrals. However, since the sum is over a   set on integers  $(K_1, K_2, K_3)\in (\Z^n_L)^3$  which are the zeros of the quadratic form $\Omega_3$,  this leads to deep problems in analytic number theory. 

%%%%JALAL
\underline{The Circle Method.} Our main tool to rigorously perform the approximation mentioned above and to estimate the resonant sum is based on recent improvements of the Hardy-Littlewood circle method (e.g.\ smooth versions of Kloosterman's leveling method as presented by Hooley \cite{Hooley} and by  Duke, Friedlander, and Iwaniec  \cite{DFI}).  Our treatment of  the circle method,   which is presented in Section \ref{sec:circle},   follows closely the work of Heath-Brown \cite{HB}. However, the results in \cite{HB} are not applicable to the PDE problem we are considering here.  Namely we need to estimate functions defined as weighted lattice sums.  More specifically, we consider,
\[
F(x) = \sum_{Q(K)=0} W\left(\mathcal{L}(x,K)\right),
\]
where $Q$ is a quadratic form and $\mathcal L$ is a linear function, and ask for {\it 1)} boundedness of the map $W\to F$, in certain function spaces;  and {\it 2)} the convergence of $F$ (in  function space) as $L\to \infty$.  In addition the function $W$ is neither compactly supported nor does it vanish when $Q(K)=0$.

Using the circle method, we prove that  there exists a normalization factor $Z_n(L)$, such that the resonant sum converges to an explicit integral operator $\mathcal{T}$ as follows: If $f$ is sufficiently smooth and decaying, then 
$$
\frac{1}{Z(L)} \mathcal{T}_L(f,f,f) \overset{L \to \infty}{\longrightarrow} \mathcal{T}(f,f,f) \quad \text{with}\quad Z_n(L) = \left\{
\begin{array}{ll}
\frac{1}{\zeta(2)} L^2 \log L & \mbox{if $n=2$} \\
\frac{\zeta(n-1)}{\zeta(n)} L^{2n-2} & \mbox{if $n \geq 3$}
\end{array} \right.
$$
where $\zeta$ is the Riemann zeta function.  The operator $\mathcal T$ is an integral over the set  $\mathcal S(K) =0$ and $\Omega(K)=0$,
\[
\mathcal{T}(f,f,f)(K) = \int f(K_1) \bar f(K_2)  f(K_3)\delta_{\R^n}(\mathcal{S}_3(K))\delta_{\R}(\Omega_3(K))\, dK_1dK_2dK_3.
\]
 These results are presented in Sections  \ref{sec:res-bound}  and \ref{sec:asymp}.

Consequently the limiting dynamics of $a_K$ (up to rescaling time by a factor $\frac{L^{2n}}{Z_n(L) \epsilon^2}$) is given by the ``Continuous Resonant" equation
\begin{equation}
\label{CR}
\boxed{
\tag{CR}
- i \partial_t g(t, \xi) = \mathcal{T}\left(g(t, \cdot),g(t, \cdot),g(t, \cdot)\right) (t, \xi) \qquad \xi \in \mathbb R^n,
}
\end{equation}
provided   $\epsilon< L^{-\gamma}$ for arbitrary small $\gamma>0$. This equation is studied in the companion paper~\cite{BGHS}.

For general $p\in \mathbb{N}, n \geq 1$ with $np\geq 2$, and again by using the circle method, we can show that the resonant sum converges to an integral operator with a normalization factor 
\[
Z_{n,p}(L) = \left\{
\begin{array}{ll}
\frac{1}{\zeta(2)} L^2 \log L & \mbox{if $np=2$} \\
\frac{\zeta(np-1)}{\zeta(np)} L^{2pn-2} & \mbox{if $np \geq 3$}
\end{array} \right.
\]
Thus the term $\log L$ appears when $n=1$ and $p=2$, or when $n=2$ and $p=1$.  The case of general $p\in \mathbb N$ will be presented in Section \ref{section_general_p}.

\subsection{Main results}  Again, for ease of notation and presentation,  we first present our results for $p=1$. For the general case, i.e.,  $p\in\mathbb{N}$, the results and sketch of the modifications of the proofs are presented in Section \ref{section_general_p}. 

The notations needed  for the statements of the theorems  below are given at the end of the introduction.   The first theorem gives the time-scale and rate of convergence of solutions  of the~\eqref{NLS} to those of the~\eqref{CR} equation.

%\comment{(Tristan) The regularity conditions in the below theorems need to be fixed.}

\begin{theorem}\label{approx thm1}
Fix $\ell> 2n$ and $0<\gamma<1$. Let $g_0\in X^{\ell+n+2, 3n+3}(\RR^n)$, and suppose that $g(t,\xi)$ is a solution of~\eqref{CR} over a time interval $[0, M]$ with initial data $g_0=g(t=0)$. Denote by
$$
B\overset{def}{=}\sup_{t\in [0, M]}\|g(t)\|_{X^{\ell+n+2, 3n+3}(\RR^n)}.
$$
Let $u$ be the solution of~\eqref{NLS} with initial data $u_0 = \frac{1}{L^n} \sum_{\mathbb{Z}^n_L} g_0(K) e(K\cdot x)$, and set for $K \in \mathbb{Z}^n_L$
$$
a_K(t) = \widehat{u}_K(t) e(-|K|^2 t).
$$
Then for $L$ sufficiently large, and  $\epsilon^2L^\gamma$ sufficiently small, there exists  a  constant
$C_{\gamma, M, B}$ such that or all $t\in [0, MT_R]$,
\begin{equation} \label{tildeak} 
\left\| a_K(t)-g\left(\frac{t}{T_R},K\right)\right\|_{X^\ell(\mathbb{Z}_L^n)}\lesssim C_{\gamma, M, B} \left( \delta(L) +\epsilon^2 L^\gamma \right) ,
\end{equation}
where 
\begin{equation}\label{eq: Z_n}
T_R=\frac{L^{2n}}{\epsilon^2 Z_n(L)}, \qquad
Z_n(L)= \begin{cases} L^{2n-2} \frac{\zeta\left(n-1\right)}{\zeta\left(n\right)} \qquad &\text{ if }n \geq 3,\\
					L^{2n-2} \frac{\log L}{\zeta(2)}\qquad &\text{ if }n =2.
		\end{cases}
\end{equation}
and
\begin{equation}
\delta(L)=
\begin{cases} L^{-1 +\gamma} \qquad &\text{ if }n\geq 3\\
(\log L )^{-1} \qquad &\text{ if } n=2.
\end{cases}
\end{equation}
\end{theorem}

A few remarks are in order. Firstly, the (CR) equation is locally well-posed in the spaces $X^{\ell, N}$ for any $\ell > d-1$. This is proved in the companion paper~\cite{BGHS}. Secondly, when $n=2$, this result was first proved in \cite{FGH} under the restriction $\epsilon< L^{-1-\gamma}$. Note that in this case the rate of convergence is $(\log L)^{-1}$.  Here we improve this result  to obtain polynomial rate in  $L$.   This is accomplished by  using the circle method to identify  the logarithmically decaying correction term, which allows to state an approximation result with polynomially decaying error term in $L$. Thus we have the following:

\begin{theorem}\label{approx thm2}
Let $n=2$. Fix $\ell>4 $ and $0<\gamma<1$. Let $g_0\in X^{\ell+6, 15}(\mathbb{R}^2)$ and suppose that $g(t,\xi)$ is a solution of 
\begin{equation}\label{modified CR}
-i\partial_t g(t,\xi)=\mathcal T(g,g,g) +\frac{\zeta(2)}{\log L} \mathcal C(g, g, g)
\end{equation}
over a time interval $[0, M]$ with initial data $g_0=g(t=0)$, and denote by
$$
B\overset{def}{=}\sup_{t\in [0, M]}\|g(t)\|_{X^{\ell+6, 15}(\mathbb{R}^2)}.
$$
Let $u$ be the solution of~\eqref{NLS} with initial data $u_0 = \frac{1}{L^2} \sum_{\mathbb{Z}^2_L} g_0(K) e(Kx)$, and set for $K \in \mathbb{Z}^2_L$
$$
a_K(t) = \widehat{u}_K(t) e(-|K|^2 t).
$$
Then for $L$ sufficiently large, and  $\epsilon^2L^\gamma$ sufficiently small, there exists  a  constant
$C_{\gamma, M, B}$ such that or all $t\in [0, MT_R]$,
\begin{equation}\label{tildea_K and g d=2} 
\begin{split}
\left\| a_K(t)- g\left(\frac{t}{T_R},K\right)\right\|_{X^\ell(\mathbb{Z}_L^2)}\lesssim&  C_{\gamma, M, B} \left(\frac{1}{L^{1/3-\gamma}}+\epsilon^2  L^\gamma \right),
\end{split}
\end{equation}
where $T_R\overset{def}{=}\frac{\zeta(2) L^{2}}{\epsilon^2 \log L}$.
\end{theorem}
\medskip

\begin{remark}
All theorems in this paper, including the above two, have analogs for solutions (respectively sequences) on the unit torus $\mathbb{T}^n$  (respectively $\Z^n$). There, equation (CR) gives the effective dynamics for high-frequency envelopes of solutions of the (NLS) equation \eqref{pNLS} posed the unit torus $\mathbb{T}^n$. More precisely, starting with a solution $g(t)$ of the (CR) equation (as in the above two theorems) and taking initial data for \eqref{pNLS} (with $L=1$) of the form $\widehat u_0(k)=C_N g_0(\frac{k}{N})$ for each $k \in \Z^n$, where $C_N$ is a proper normalization parameter, one can show that $e^{-it|k|^2}\widehat u(t, k)$ is approximated (in the limit of large $N$) by $C_N g(\frac{t}{T_N}, \frac{k}{N})$ for an appropriate time-scale $T_N$. This follows from a direct rescaling of the above two theorems; we refer to \cite{FGH}[Theorem 2.6] and its proof for the details of this rescaling argument.
\end{remark}

\subsection*{Notations} Throughout the paper we adopt the notation: $e(z)=e^{2\pi i z}$ for any $z\in \mathbb C$, and the following normalization for the Fourier transform of a function $f$ on $\mathbb{R}^d$:
$$
\mbox{if $\xi \in \mathbb{R}^d$}, \quad \mathcal{F} f (\xi) = \widehat{f} (\xi) = \int_{\mathbb{R}^d} e(-x \cdot \xi) f(x)\,dx.
$$
The inverse Fourier transform reads then
$$
\mathcal{F}^{-1} g (x) = \check{g} (x) =  \int_{\mathbb{R}^d} e(x \cdot \xi)  g(\xi)\,d\xi.
$$
The Fourier transform of a function $f$ on the torus $\mathbb{T}^n_L = [0,L]^n$ is given by
$$
\mbox{if $K \in \mathbb{Z}^n_L = \left( \frac{\Z}{L} \right)^n$}, \quad \mathcal{F} f (K) = \widehat{f}_K = \int_{\mathbb{T}^d_L} f(x) e(- K \cdot x) \,dx,
$$
while the inverse Fourier transform reads
\[
[\mathcal{F}^{-1} g] (x) = \frac{1}{L^d} \sum_{K \in \mathbb{Z}^d_L} g_K e(K  \cdot x) .
\]

For multilinear sums and quadratic forms arising from \eqref{NLS}, we denote by:
\begin{itemize}
\item $\mathbb{Z}_L = \mathbb{Z} / L$. 
\item $\mathcal S_3(K) = S(K,K_1,K_2,K_3) = K_1 - K_2 + K_3 - K$.
\item $\mathcal S_{2d+1}(K) = S(K,K_1,\dots,K_{2d+1}) = K_1 - K_2  + \dots - K_{2d} + K_{2d+1} - K$.
%\item $\mathcal{S}_3(K) = \{ (K_1,K_2,K_3)\in ((\mathbb{Z}_L)^d)^3 \; \mbox{such that} \; S(K)=0 \}$
%\item $\mathcal{S}_{2d+1}(K) = \{ K_i \in \mathbb{Z}_L \; \mbox{such that} \; S_{2d+1}(K)=0 \}$
\item $ \Omega_3(K) =  \Omega_3(K,K_1,K_2,K_3)  = |K_1|^2 - |K_2|^2 + |K_3|^2 -|K|^2$.
\item $\Omega_{2d+1}(K) =  \Omega_{2d+1}(K,K_1,\dots, K_{2d+1}) =  |K_1|^2 - |K_2|^2 + \dots - |K_{2d}|^2+ |K_{2d+1}|^2 -|K|^2$
\end{itemize}

Given two quantities $A$ and $B$, we denote
\begin{itemize}
\item  $\displaystyle \<A\> = \sqrt{1 +A^2}$.  
\item  $\displaystyle A \lesssim B \quad   \mbox{if  } \exists \; C,   \mbox{ a constant, such that  } A \leq CB$.
\item  $\displaystyle A \lesssim L^{\alpha +} B    \quad  \mbox{if  } \forall  \delta>0,  \; \exists \; C_\delta \mbox{ (independent of $L$) such that }  A \leq C_\delta L^{\alpha + \delta} B$. 
\item  $\displaystyle A \ll B \; \mbox{if }  A \leq cB \; \mbox{for a sufficiently (depending on the context) small constant}\;  c>0 $.
\end{itemize}

Our functional framework will be based on the $X^\ell$ spaces defined by the norm
$$
\| f \|_{X^\ell} = \| \langle x \rangle^\ell f(x) \|_{L^\infty}
$$
(which is meaningful whether $f$ is defined on $\mathbb{R}^n$ or $\mathbb{Z}^n_L$). If $f$ is defined on $\mathbb{R}^d$, a variant is given by the $X^{\ell,N}$ spaces defined by the norm
\begin{equation}\label{def of Xln}
\|f\|_{X^{\ell,N}(\mathbb{R}^d)}=\sum_{0\leq |\alpha|\leq N} \|\nabla^\alpha f\|_{X^\ell}.
\end{equation}

We also recall Abel's summation formula: If $\phi \in C^1([1, \infty))$ and $x\in \mathbb{N}$, 
\begin{equation}\label{e:Abel_summation}
\sum_{1\leq n\leq x} a_n\phi(n)= A(x)\phi(x)-\int_1^{x}A(u)\phi'(u)du,
\end{equation}
where 
\[A(x):= \sum_{1\leq n\leq x} a_n.\]

{\bf Acknowledgments} The authors are very grateful to Peter Sarnak for many most helpful discussions on delicate number theory questions.

TB was supported by the National Science Foundation  grant DMS-1600868.
PG was supported by the National Science Foundation  grant  DMS-1301380.
JS was  supported by the National Science Foundation  grant DMS-1363013.
ZH was supported by National Science Foundation grant DMS-1600561, a Sloan Fellowship, and a startup fund from Georgia Institute of Technology. 

\section{Lattice Sums and the circle method}\label{sec:circle}
An effective method to obtain quantitative estimates on approximating Lattice sums by integrals is the Circle Method.  Amongst the analytical tools used in this method are the Fourier transform or series, and the Poisson summation formula. A simple illustrative example on how these tools figure in obtaining quantitative estimates on lattice sums is to consider a 1-periodic function $g$, and express it in terms of its Fourier series

\[
g(x) = \sum_{n = -\infty}^{\infty} \widehat{g}(n)\,e(nx), \quad \mbox{ where }
e(t)\overset{def}{=}e^{2\pi it},
\]
and observe that since ($L \in \mathbb N$ here)
\[
\sum_{a=0}^{L-1} e\left(\frac{an}{L}\right) = 
\begin{cases}
L &\quad \text{if } L\mid n,\\
0 &\quad \text{if } L\nmid n
\end{cases}
\]
we have 
\[
\sum_{a=0}^{L-1} g\left(\frac{a}{L}\right) = L \sum_{k=-\infty}^\infty
\widehat{g}(kL).
\] 

By isolating the zeroth Fourier mode, we conclude that
\begin{equation}
\label{rhr}
\frac{1}{L} \sum_{a=0}^{L-1} g \left(\frac{a}{L}\right) - \int_0^1 g(x)
dx = \sum_{k\neq 0} \widehat{g} (kL).
\end{equation}
Therefore, if $g$ is  a smooth function, the right hand side of the equation above decays quickly in $L$:
\[
 \sum_{k\neq 0} \widehat{g} (kL) = O\left(\frac 1{L^N}\right)
\]
for any $N$.

\subsection{The Circle Method}  Let $ Q$ be a non-degenerate  quadratic form with integer coefficients defined on $\mathbb{R}^d$, and for any $m\in \mathbb{Z}$ let $ Q_{m}(z) =  Q(z) -m$.  For a fixed $\mu \in \mathbb{R}$ and $L>0$ large such that ${\mu L^2}\in \mathbb{Z}$, we form the sum,
\[
\sum_{\substack{ Q_{\mu L^2}(z) = 0\\z\in \mathbb{Z}^d}} W\left(\frac{z}{L}\right)
\]
where $W$ is a smooth rapidly decreasing function on $\mathbb{R}^d$.  

Our aim is to obtain good asymptotic of this sum for  large $L>0$. To do so, first we localize $W$ close to $ Q_{\mu L^2}(z) = 0$ by introducing an even cut off function $\chi \in C_0^\infty
\left(-\frac{1}{2}, \frac{1}{2}\right)$, $\chi (0) = 1$,  and write the sum as
\[
\sum_{\substack{ Q_{\mu L^2}(z) = 0\\z\in \mathbb{Z}^d}} W\left(\frac{z}{L}\right)=
\sum_{\substack{ Q_{\mu L^2}(z) = 0\\z\in \mathbb{Z}^d}} W\left(\frac{z}{L}\right) \chi\left( Q_{\mu L^2 }\left(\frac zL\right)\right)
\overset{def}{=}
\sum_{\substack{ Q_{\mu L^2}(z) = 0\\z\in \mathbb{Z}^d}} W_\chi \left(\frac{z}{L}\right),
\]
and  note that
\begin{equation}\label{sum}
\sum_{\substack{ Q_{\mu L^2}(z) = 0\\z\in \mathbb{Z}^d}} W \left(\frac{z}{L}\right) = \sum_{z \in \mathbb{Z}^d}
W_\chi \left(\frac{z}{L}\right) \int_0^1 e(\alpha  Q_{\mu L^2}(z)) d\alpha = \int_0^1  \sum_{z \in \mathbb{Z}^d} W_\chi \left(\frac{z}{L}\right)e(\alpha  Q_{\mu L^2}(z)) d\alpha.
\end{equation}
For rational $\alpha = \frac aq$, the integrand $\ \sum_{z \in \mathbb{Z}^d} W_\chi \left(\frac{z}{L}\right)e(\alpha  Q_{\mu L^2}(z))$ depends critically on the size of the denominator $q$. This can be illustrated by the following elementary calculation
\[
\sum_{\ell=1}^L e\left(\frac aq \ell^2\right) \sim \sum_{m=0}^{[\frac Lq]}  \sum_{i= 0}^{q-1} e\left(\frac aq(i+ mq)^2 \right) \sim  \frac Lq  \sum_{i= 0}^{q-1} e\left(\frac aqi^2 \right) \sim \frac L{\sqrt{q}}
\]
where in the last calculation we used the formula for Gauss's sum. From this calculation one can deduce two things:  1) For $q>L$ the contribution to the sum in \eqref{sum} should be of order $O(L^{\frac d2})$; and 2)  The major contribution to the sum comes from intervals around rational points whose denominator is small, i.e., $q<L$. Thus to approximate the sum it pays off to split the integral  by introducing  a $1$-periodic  partition of unity
which separates rational numbers according to the size of their denominator relative to the scale $L$.  For this, we first notice that any rational point $\frac aq$,  where the $\mbox{GCD}$ of $a$ and $q$ is $1$, i.e. $(a,q) =1$, is a distance of order $O(\frac 1{qL})$ from  rational points with denominators smaller than $L$.  Thus it is reasonable to look for a partition of unity generated by 1-periodic functions, centered around rational points $\frac aq$, concentrated on an interval of size  $O(\frac 1{qL})$. To this effect, define the family of functions $\psi$  indexed by  $q$ (or $\frac qL$ since the contribution depends on this ratio and $L$ is fixed once and for all), parametrized by $L$ to be such that 
\[
\mbox{for all $\alpha \in \mathbb{R}$,} \quad \sum_{q=1}^\infty \sum_{\substack{a=0 \\(a,q) = 1}}^{q-1} \psi\left(\alpha -
\frac{a}{q}; \frac{q}{L}, L\right) = 1,
\]
with the understanding that $(0,q)=1$ if and only if $q=1$.  

The 1-periodic functions $\psi$ will be constructed from a smooth real valued even  function $\phi$, defined on $\mathbb{R}\times(0,\infty)\times(0,\infty)$, by
\[
\psi\left(\beta; \frac{q}{L}, L\right) = \sum_{\ell \in \mathbb{Z}} \phi\left(qL(\beta + \ell); \frac{q}{L}, L\right),
\]
in which case the desired partition of unity can be written as
\[
 \sum_{q=1}^\infty \sum_{\substack{a=0 \\(a,q) = 1}}^{q-1} \sum_{\ell \in \mathbb{Z}} \phi\left(qL(\alpha - \frac aq + \ell); \frac{q}{L}, L\right) = 1.
\]

Writing the Fourier series  coefficients of the above identity, we obtain
\begin{equation}
\label{deltan}
\sum_{q=1}^\infty \sum_{\substack{a=0 \\(a,q) = 1}}^{q-1} e
\left(-\frac{a}{q} n\right) \frac 1{qL} \widehat{\phi}\left(\frac{n}{qL}; \frac{q}{L}, L\right)  = \delta_n
\end{equation}
where $\widehat\phi$ is the Fourier transform of $\phi$ in the first variable, and $\delta_n$ is Kronecker delta.  
 
 The existence of such a $\phi$,  which is far from obvious, was established in \cite{DFI}, and will be presented  below  following  \cite{HB}. We will be able to ensure that $\widehat \phi$ is real-valued which leaves identity \eqref{deltan} invariant under taking conjugates (or equivalently switching the sign of $a$). So assuming for the moment that such a $\widehat \phi$ exists, we can write 
\[
\int_0^1 e(\alpha  Q_{\mu L^2}(z))d\alpha = \delta_{ Q_{\mu L^2}(z)} = \sum_{q=1}^\infty
\sum_{\substack{a=0 \\(a,q) = 1}}^{q-1} e
\left(\frac{a}{q}  Q_{\mu L^2}(z)\right) \frac 1{qL} \widehat{\phi}\left(\frac{ Q_{\mu L^2}(z)}{qL}; \frac{q}{L}, L\right),
\]
and consequently the original sum can be written as
\[
\sum_{\substack{ Q_{\mu L^2}(z) = 0\\z \in \mathbb{Z}^d}}
W\left(\frac{z}{L}\right) = \sum_{\substack{z\in \mathbb{Z}^d\\q \geq
    1}} \sum_{\substack{a=0 \\(a,q) = 1}}^{q-1}
W_\chi\left(\frac{z}{L}\right) \frac 1{qL} \widehat{\phi}\left(\frac{ Q_{\mu L^2}(z)}{qL}; \frac{q}{L}, L\right) e\left(\frac{a Q_{\mu L^2}(z)}{q}\right).
\]

For a fixed rational point $\frac aq$ we use Poisson summation formula to approximate each of these sums by an integral.  Recall that  for any $b\in \mathbb{Z}^d$, the Poisson summation formula is given by 
\[
\sum_{y \in \mathbb{Z}^d} g(b + qy) =\frac 1{q^d} \sum_{c \in \mathbb{Z}^d} e\left(\frac{b \cdot
  c}{q}\right) \int_{\mathbb{R}^d} g(x) e\left(-\frac{c \cdot x}{q}\right)\, dx.
\]
Writing now
\begin{multline*}
 \sum_{z\in
  \mathbb{Z}^d} W_\chi \left(\frac{z}{L}\right) \frac 1{qL} \widehat{\phi}\left(\frac{ Q_{\mu L^2}(z)}{qL}; \frac{q}{L}, L\right) e\left(\frac{a Q_{\mu L^2}(z)}{q}\right) =\\
 \sum_{\substack{y \in \mathbb{Z}^d \\[.3em] 0\le b_i\le q-1\\[.3em] 1 \leq i \leq d}}W_\chi \left(\frac{b + qy}{L}\right) \frac 1{qL} \widehat{\phi}\left(\frac{ Q_{\mu L^2}(b + qy)}{qL}; \frac{q}{L}, L\right) e\left(\frac{a Q_{\mu L^2}(b + qy)}{q}\right)
\end{multline*}
and noting that
\[
e\left(\frac{a Q_{\mu L^2}(b + qy)}{q}\right) = e\left(\frac{a Q_{\mu L^2}(b)}{q}\right),
\]
we can apply Poisson summation formula to conclude that
\begin{equation*}
\sum_{\substack{ Q_{\mu L^2}(z) = 0\\z \in \mathbb{Z}^d}}
W\left(\frac{z}{L}\right)=
\sum_{q=1}^\infty \sum_{c \in \mathbb{Z}^d} S(q,c) \frac 1{q^d}
\int_{\mathbb{R}^d} W_\chi \left(\frac{x}{L}\right) \frac 1{qL} \widehat{\phi}\ \left(\frac{ Q_{\mu L^2}(x)}{qL}; \frac{q}{L}, L\right) e\left(-\frac{c \cdot x}{q}\right)\,dx
\end{equation*}
where 
\begin{equation*}
%\label{SLC}
S_{\mu L^2}(q,c) =  \sum_{\substack{a=0 \\(a,q)=1}}^{q-1}\sum_{\substack{b_i =0 \\ 1\le i \le d}}^{q-1} e \left(\frac{a Q_{\mu L^2}(b) + c\cdot b}{q}\right).
\end{equation*}

By  changing variables $x\mapsto Lx$ and denoting by  $r=\frac qL$, the sum can be written as 
\begin{equation}\label{sumpsi}
\sum_{\substack{ Q_{\mu L^2}(z) = 0\\z \in \mathbb{Z}^d}}\!\!\!\!\!
W\left(\frac{z}{L}\right) = 
L^{d-2} \sum_{q=1}^\infty \sum_{c \in \mathbb{Z}^d} S_{\mu L^2}(q,c)  \frac 1{q^d}\int_{\mathbb{R}^d}
W_\chi(z) \frac 1{r} \widehat{\phi}\ \left(\frac{ Q_{\mu}(x)}{r}; r, L\right) e\left(-\frac{c \cdot x}{r}\right)\,dx
\end{equation}

\subsection{Existence of the partition of unity} The partition of Kronecker delta as in~\eqref{deltan} is essentially due to Duke, Friedlander and Iwaniec \cite{DFI} (see also \cite{IK}).
Here we follow the formulation of Heath-Brown \cite{HB}. 

%%%%Jalal
In order to simplify the sum in \eqref{deltan}, we adopt the ansatz
\[
 \frac 1{qL} \widehat{\phi}\left(\frac{n}{qL}; \frac{q}{L}, L\right)  =  \sum_{j=1}^\infty  \frac 1{jqL} v\left(\frac{n}{jqL}; \frac{jq}{L}, L\right)
 \]
(for a function $v$ to be defined). Using the fact that for any function $u$
 \[
\sum_{k=1}^\infty \sum_{\alpha=0}^{k-1} u(\alpha,k) =\sum_{j=1}^\infty \sum_{q = 1}^\infty 
\sum_{\substack{a =0 \\(a,q) = 1}}^{q-1} u(ja, jq),
\]
equation \eqref{deltan}  can be written in terms of $v$ as 
\begin{align}
\sum_{q=1}^\infty \sum_{\substack{a=0 \\(a,q) = 1}}^{q-1}  & e\left(-\frac{a}{q} n\right) \frac 1{qL} \widehat{\phi}\left(\frac{n}{qL}; \frac{q}{L}, L\right) \\
& = \sum_{q=1}^\infty \sum_{\substack{a=0 \\(a,q) = 1}}^{q-1} e\left(-\frac{a}{q} n\right)\sum_{j=1}^\infty  \frac 1{jqL} v\left(\frac{n}{jqL}; \frac{jq}{L}, L\right) \\
& = \sum_{k=1}^\infty \sum_{\alpha=0}^{k-1} e\left(-\frac{\alpha}{k} n\right)\frac 1{kL} v\left(\frac{n}{kL}; \frac{k}{L}, L\right) \\
& =\sum_{k\mid n}    \frac 1{L} v\left(\frac{n}{kL}; \frac{k}{L}, L\right) = \delta_n.
\end{align}
The function $v$ was constructed in \cite{DFI} as follows. 
Let $\omega_0 \in C_0^\infty (\frac12,1)$, $\omega_0 \geq 0$
\[
\int_{\frac12}^1 \omega_0 (z) dz = 1,
\]
then for $n \in \mathbb{Z}$
\[
\sum_{k>0,\; k\mid n} \left[ \omega_0 \left(\frac{k}{L}\right) - \omega_0
\left(\frac{|n|}{kL}\right) \right] = 
\begin{cases}
0 \quad &n\neq 0\\
\sum_{k=1}^{\infty} \omega_0 (\frac{k}{L}) \quad &n=0
\end{cases}
\]
Note that by \eqref{rhr} we have
\[\frac{1}{L} \sum_{k=1}^\infty \omega_0 \left(\frac{k}{L}\right) =
\int_{-\infty}^{\infty} \omega_0(z)\,dz + O\left(\frac{1}{L^N}\right)\]
for any $N > 0$, therefore
\[
\sum_{k=1}^\infty \omega_0 \left(\frac{k}{L}\right) = LC_L;  \quad C_L = 1 +
O\left(\frac{1}{L^N}\right),
\]
and
\[
\begin{split}
 \delta_n &= \frac{1}{LC_L}\sum_{k\mid n} \left[ \omega_0 \left(\frac{k}{L}\right) - \omega_0\left(\frac{|n|}{kL}\right) \right]\\
 &= \frac{1}{LC_L}\sum_{k\mid n} \left[ \omega_0 \left(\frac{k}{L}\right) - \omega_0
\left(\frac{|n|}{kL}\right) \right],
\end{split}
\]
where the cutoff function $\chi$ (which is chosen to be even, smooth, compactly supported on $[-\frac{1}{2},\frac{1}{2}]$, and such that $\chi(0)=1$) is introduced here for convenience. Consequently we can take $v$ to be 
\[
v\left(\frac{n}{kL}; \frac{k}{L}, L\right) =\frac{1}{C_L}\left[\omega_0 \left(\frac{k}{L}\right) - \omega_0
\left(\frac{|n|}{kL}\right)\right],
\]
and thus
\begin{align*}
\frac 1{qL}\widehat{\phi} \left(\frac{n}{qL}; \frac qL, L\right)  &=   \frac{1}{C_L} \sum_{j=1}^{\infty}
\frac{1}{qjL} \left[\omega_0 \left(\frac{qj}{L}\right) -
\omega_0\left(\frac{|n|}{qjL}\right)\right].
\end{align*}
We can simplify this expression by letting
$y= \frac n{L^2}$, and define $h$ by
\begin{equation}
h(r,y)  \overset{def}{=}  \sum_{j=1}^{\infty}
\frac{1}{rj} \omega_0 (rj) - \omega_0\left(\frac{|y|}{jr}\right),
\end{equation}
(recall that $r=\frac{q}{L}$) so that we can express $\widehat\phi$ as
\[
\begin{split}
\frac 1{r}\widehat{\phi}\left(\frac yr;r,L\right)=\frac{1}{C_L}  h(r,y).   % \overset{def}{=}  \frac 1{C_L}  h_\chi(r,y).
\end{split}
\]

Going back to \eqref{sumpsi}, we notice that we have the freedom of multiplying $h(r, y)$ by $\chi (y) $ at no cost (for example by starting with $W_{\chi^2}$ instead of $W_\chi$ which also leaves the original sum invariant). Therefore, with 
$$
h_\chi(r, y)\overset{def}{=}  \frac {\chi(y)}{C_L}  h(r,y),
$$
we can summarize these calculations in the following theorem:
\begin{theorem}\label{th:cm}
\begin{equation}
\label{CM}
\sum_{\substack{ Q_{\mu L^2}(z) = 0\\z\in \mathbb{Z}^d}} W\left(\frac{z}{L}\right)= \frac{L^{d-2}}{C_L}
\sum_{q=1}^\infty \sum_{c \in \mathbb{Z}^d} S_{\mu L^2}(q,c) \frac 1{q^d} I_\mu(r, c)
\end{equation}

\begin{equation}
\label{SLC}
S_{\mu L^2}(q,c) =  \sum_{\substack{a=0 \\(a,q)=1}}^{q-1}\sum_{\substack{b_i =0 \\ 1\le i \le d}}^{q-1} e \left(\frac{a Q_{\mu L^2}(b) + c\cdot b}{q}\right)
\end{equation}

\begin{equation}
\label{ICR}
\begin{split}
I_\mu(r,c) &= \int_{\mathbb R^d} W_\chi(x)h_\chi\left(r, Q_{\mu}(x)\right)  e\left(-\frac{c \cdot x}{r}\right)\, dx
\end{split}
\end{equation}

where $r = \frac qL$ and $C_L = 1 + O(\frac{1}{L^N})$ for all $ N > 1$.  

\end{theorem}

The integral $I_\mu(r,c)$ can be expressed in terms of the Fourier transform of $h_\chi$ in the following manner.   Writing
\[
\widehat{h}_\chi(r,s) = \int_{-\infty}^{\infty} h_\chi(r,y)e(-sy)dy,
\]
then 
\[
 I_\mu(r,c) = \int_{-\infty}^{\infty} \widehat{h}_\chi(r,s) \int_{\mathbb R^d} W_\chi(x) e{\left(s Q_{\mu}(x) -\frac{c \cdot x}{r}\right)}\, dx\, ds.
\]

\begin{remark 1}
At this point we remark that Theorem \ref{th:cm}  will be used to obtain sharp upper bound on the lattice sum for any $\mu$ as well as the asymptotics of the resonant sum, which corresponds to $\mu=0$.  When $\mu=0$, we simplify the notation by writing $I$ and $S$ for $I_0$ and $S_0$ . Finding the asymptotics of the lattice sum for large $L$ is equivalent to finding good asymptotics for the arithmetic function $S$ and the integral $I$. In order to prove Theorems \ref{approx thm1} and \ref{approx thm2}, the weight $W$ will be chosen to depend on $K$ in a nontrivial way. The dependence of $W$ on $K$ will manifest itself in two ways: 
\begin{enumerate}
\item
One has to obtain bounds on the resonant sum that are sharp both in terms of the dependence on $L$ and in terms of the function space where the approximation is performed.
\item
One has to prove the asymptotics of the resonant sum while keeping track of the size of the error in the relevant function space.  
\end{enumerate}

We will give below the asymptotics of $h$ and $S$,  and leave the analysis taking the weight $W$ into account to the following two sections.
\end{remark 1}

%may24 2

\subsection{The function $h_\chi$ and its Dirac mass asymptotics}  
To find the asymptotics of $I_\mu(r,c)$ we have to understand the behavior of the function $h_\chi$ as $r\to 0$. In fact we will see below that  as $r\to0$, $h_\chi(r, \cdot)\to \delta$.

Recall that $h_\chi :\mathbb{R}_+\times \mathbb{R} \to \mathbb{R}$ is defined as 
\[
h_\chi(r,y) = \chi(y) h(r,y) =\chi(y) \sum_{j=1}^\infty \left[ \frac{1}{rj} \omega_0(rj) - \frac{1}{rj} \omega_0 \left( \frac{|y|}{rj} \right) \right].
\]
where $\chi$ is  even, $\chi \in C_0^\infty
\left(-\frac{1}{2}, \frac{1}{2}\right)$, $\chi (0) = 1$,  and   $\omega_0 \in C_0^\infty
\left(\frac{1}{2}, 1\right)$  with  $\int_0^\infty \omega_0 = 1$. Thus $h_\chi$ is even in $y$ and
\[
\supp h_\chi \subset   (0,1) \times   \left(-\frac{1}{2}, \frac{1}{2}\right).
\]
Consequently we need to analyze $h$ for $(r,y)\in  (0,1) \times   \left(-\frac{1}{2}, \frac{1}{2}\right)$.
\begin{lemma}
\label{hummingbird}
For  $(r,y)\in  (0,1) \times   \left(-\frac{1}{2}, \frac{1}{2}\right)$,
\begin{itemize} 
\item[(i)] If $|y| < \frac{r}{2} $, $h$ is independent of $y$ and $\displaystyle |\partial_r^k h(r,y)| \lesssim \frac{1}{r^{k+1}}$ for all $k \geq 0$.
\item[(ii)] $\displaystyle h(r,y) = \frac{1}{r} \overline h (r) - \frac{1}{r} \widetilde h \left( \frac{r}{|y|} \right)$ with 
\[
|\partial_z^k \overline h(z) | + |\partial_z^k \widetilde h(z) | \lesssim |z|^N,   \quad \forall k,N \ge 0.
\]

Consequently, one has the estimate $\sup_{r} \|h\|_{L^1_y}\lesssim 1$ and $\sup_{r} r\|\partial_r h\|_{L^1_y}\lesssim 1$.

 \item[(iii)] For any $1>a>r>0$, and $n\in \mathbb{N}_0$ we have
 \begin{equation}\label{eq:appxid}
\int_{-a}^a x^n h(r,x)\, dx = \delta_n + O(r^N),
\end{equation}
for any positive $N$.
\end{itemize}
\end{lemma}

\begin{proof}
The proof of $(i)$ is straightforward, so we move on to the proof of $(ii)$.
By the Poisson summation formula,
\begin{equation}
\label{pinguin1}
\begin{split}
\sum_{j \in \mathbb{Z}} \frac{1}{r j} \omega_0(r j) & = \frac{1}{r} \sum_{\ell \in \mathbb{Z}} \int \frac{1}{s} \omega_0(s) e \left( \frac{-\ell s}{r} \right) \,ds \\
& = \frac{1}{r} \int \frac{1}{s} \omega_0(s)\,ds + \frac{1}{r} \underbrace{\sum_{\ell \neq 0} \int \frac{1}{s} \omega_0(s) e \left( \frac{-\ell s}{r} \right)}_{\displaystyle \overline h(r)} \,ds
\end{split}
\end{equation}
It is easy to check, by repeated integration by parts, that $ |\partial_r^k \overline{h} (r)| \lesssim r^N$ for all $k \geq 0$.

Another application of Poisson's summation formula gives
\begin{equation}
\label{pinguin2}
\begin{split}
\sum_{j \in \mathbb{Z}} \frac{1}{r j} \omega_0\left( \frac{|y|}{rj} \right) & = \frac{1}{r} \sum_{\ell \in \mathbb{Z}} \int \frac{1}{s} \omega_0 \left( \frac{1}{s} \right) e \left( \frac{-\ell s|y|}{r} \right) \,ds \\
& =  \frac{1}{r}  \int \frac{1}{s} \omega_0 \left( \frac{1}{s} \right)\,ds + \frac{1}{r} \underbrace{\sum_{\ell \neq 0} \int \frac{1}{s} \omega_0 \left( \frac{1}{s} \right) e \left( \frac{-\ell s |y|}{r} \right)}_{\displaystyle \widetilde{h} \left(\frac{r}{|y|}\right)} \,ds
\end{split}
\end{equation}
Once again, repeated integrations by parts show that $|\partial^k_z \widetilde{h}(z)| \lesssim |z|^N$.

Subtracting~\eqref{pinguin2} from~\eqref{pinguin1} and noticing that $\frac{1}{r}  \int \frac{1}{s} \omega_0 \left( \frac{1}{s} \right)\,ds = \frac{1}{r} \int \frac{1}{s} \omega_0(s)\,ds$ , we see that
\begin{equation}
\label{pinguin3}
h(r,y) = \frac{1}{r} \overline h (r) - \frac{1}{r} \widetilde h \left( \frac{r}{y} \right),
\end{equation}
and the desired bound follows from the estimates on $\overline h$ and $\widetilde h$.

Next we  turn to proving $(iii)$.   Starting from the equation~\eqref{pinguin3},   we have for  $n=0$
\[
\int_{-a}^a h(r,x)\, dx  = 2\frac{a}{r} \overline{h}(r) - 2\frac{1}{r} \int_0^a \widetilde{h} \left( \frac{r}{x} \right)\,dx,
\]
since $h$ is even.  In order to estimate $\int_0^a \widetilde{h} \left( \frac{r}{x} \right)\,dx$, we rewrite it as the $\lim_{\epsilon \to 0} \int_\epsilon^a \widetilde{h} \left( \frac{r}{x} \right)\,dx$. Since the series in $\ell$ below converges uniformly for $\epsilon>0$, the following manipulations are justified:
\begin{align*}
\frac{1}{r} \int_\epsilon^a \widetilde{h} \left( \frac{r}{x} \right)\,dx & = \frac{1}{r}  \int_\epsilon^a \sum_{\ell \neq 0} \int\frac{1}{s} \omega_0 \left( \frac{1}{s} \right) e \left(\frac{ -\ell s x}{r} \right) \,ds\,dx \\
&  =\frac{1}{r}  \sum_{\ell \neq 0} \int \frac{1}{s} \omega_0  \left( \frac{1}{s} \right) \int_\epsilon^a  e \left( \frac{-\ell s x}{r} \right)\,dx \,ds \\
& =  -  \underbrace{\sum_{\ell \neq 0} \int \frac{1}{s} \omega_0  \left( \frac{1}{s} \right) \frac{1}{2\pi i \ell s} e \left( \frac{-\ell s a}{r} \right)\,ds}_{\displaystyle h_0 \left( \frac{r}{a} \right)} +  \underbrace{\sum_{\ell \neq 0} \int \frac{1}{s} \omega_0  \left( \frac{1}{s} \right) \frac{1}{2\pi i \ell s} e \left( \frac{-\ell s \epsilon}{r} \right)\,ds}_{\displaystyle J_\epsilon(r)}.
\end{align*}
Letting $\epsilon \to 0$ leads to
$$
\int_{-a}^a h(r,x)\, dx  = 2\frac{a}{r} \overline{h}(r)-  2 h_0 \left( \frac{r}{a} \right)  -2 \lim_{\epsilon \to 0} J_\epsilon(r).
$$
It is easy to show by repeated integrations by parts that $\left| \displaystyle h_0 (x) \right| \lesssim |x|^N$ for all $N$, thus in order to prove the desired result it suffices to show that $\lim_{\epsilon \to 0} J_\epsilon(r) = -\frac{1}{2}$.

 It is well known that, for $p>0$, the $p$-periodic function $\vartheta_p(x)= \frac{x}{p}-\floor{\frac xp}$ has a Fourier series
 \[
\vartheta_p(x)=  \frac{1}{2} + \sum_{\ell \neq 0} \frac{1}{2\pi i \ell} e \left(\frac{-\ell x}{p} \right) 
 \]
Therefore for a fixed $r$ and $\epsilon$ sufficiently small we have 
\[
J_\epsilon(r)=   \int \frac{1}{s^2} \omega_0  \left( \frac{1}{s} \right)  \sum_{\ell \neq 0}\frac{1}{2\pi i \ell } e \left( \frac{-\ell s \epsilon}{r} \right)\,ds =   \int \frac{1}{s^2} \omega_0  \left( \frac{1}{s} \right)\left(\vartheta_r(\epsilon s)-\frac 12\right) ds,
\]
 and since $1 <s < 2$ we conclude that 
\begin{align*}
\lim_{\epsilon \to 0}J_\epsilon(r) = \lim_{\epsilon \to 0}  \int \frac{1}{s^2} \omega_0  \left( \frac{1}{s} \right) \left(\frac{\epsilon s}{r}- \frac 12\right) \,ds = - \frac 12,
\end{align*}
which concludes the proof for $n=0$.

For $n\in \mathbb{N}$, we need to show that %since $h$ is even, we need to show 
\[
\int_{-a}^a x^n h(r,x)\, dx = \frac{\left(a^{n+1}-(-a)^{n+1}\right)}{(n+1)r} \overline{h}(r) - \frac{1}{r} \int_{-a}^ax^n \widetilde{h} \left( \frac{r}{x} \right)\,dx  = O(r^N).
\]  
Notice that the last integral is proper because $\widetilde h(x)$ is compactly supported. Since $\overline h(r) = O(r^N)$, we turn our attention to 
\[
 \int_{-a}^ax^n \widetilde{h} \left( \frac{r}{x} \right)\,dx =\int_{-a}^a x^n\left[ \sum_{\ell \neq 0} \int \frac{1}{s} \omega_0 \left( \frac{1}{s} \right) e \left( \frac{-\ell s x}{r} \right) \,ds\right]\, dx.
 \]
 First integrating $n$ times by parts in $s$ and using the compact support of $w_0$ we obtain up to constants
\[
\int_{-a}^a \sum_{\ell \neq 0} \int \frac{r^{n}}{\ell^{n}}  \left[\frac{d^n}{ds^n} \left( \frac{1}{s} \omega_0 \left( \frac{1}{s} \right)\right)\right]e \left( \frac{-\ell s x}{r} \right)\, dsdx.
\]
 Then integrating in $x$ and using the compact support of $\omega_0$  and that $n\geq 1$, we obtain up to constants
  \[
 \sum_{\ell \neq 0} \int \frac{r^{n+1}}{\ell^{n+1}}s^{-1}e \left( \frac{\pm \ell s a}{r} \right)  \left[\frac{d^n}{ds^n} \left( \frac{1}{s} \omega_0 \left( \frac{1}{s} \right)\right)\right]\, ds,
 \]
 which is of $O(r^N)$ by repeated integration by parts.
\end{proof}

From this lemma one can conclude that as $r\to 0$ the function $h_\chi(r,\cdot) \to \delta$.
\begin{corollary}\label{cor:dirac}
Let $f: (-\frac 12, \frac 12)\to \RR$ be locally Lipschitz, then
\[
\left |\int h_\chi(r,x)f(x) \, dx - f(0)\right | \lesssim r\|f\|_{Lip}.
\]
\end{corollary} 
\begin{proof}   Assume $\|f\|_{Lip}=1$.  From  Lemma \ref{hummingbird} $(iii)$,  we have 
\[
\int h_\chi(r,x)f(x) \, dx - f(0) = \int_{-\frac 12}^{\frac 12} \left[\chi(x)f(x) -f(0)\right]h(r,x)\, dx + O(r^Nf(0))
\]
Since $\chi$ is smooth with $\chi(0)=1$, $f$ is Lipschitz,  and $h$ is even, then
\[
 \int_{-\frac 12}^{\frac 12} \left|\chi(x)f(x) -f(0)\right | |h(r,x)|\, dx \lesssim  \int_{0}^{\frac r2}|xh(r,x)|\, dx 
 + \int_{\frac r2}^{\frac 12}|xh(r,x)|\, dx
 \]

By Lemma \ref{hummingbird} \emph{(i)} the first integrand is bounded by a multiple of $x/r$, and by $(ii)$ the second integrand is bounded by $\frac{x}{r}(r^N +  |r/x|^N)$.  This finishes the proof.
\end{proof}
\begin{corollary}\label{cor:Fh}
The Fourier transform of $h_\chi$ in $x$ satisfies
\begin{equation}\label{eq:hhatb}
\left |\widehat{h}_\chi(r,s)\right | \lesssim \frac 1{1+|rs|^N},
\end{equation}
for any positive integer $N$.  Moreover as $r\to 0$,  $\widehat{h}_\chi(r,s)\to 1$ uniformly on compact sets.
\end{corollary}
\begin{proof}
Since $\|h_\chi\|_{L^1_y} \lesssim 1$, then  $\|\widehat{h}_\chi\|_{L^\infty} \lesssim 1$.  Moreover, by integration by part $N$ times in $y$ on
\[
\widehat{h}_\chi(r,s) = \int  h_\chi(r,y)e(-sy)\, dy ,
\]
and using Lemma \ref{hummingbird} \emph{(ii)}, we obtain the bound in terms of 
\[
\frac 1{s^N}\int_{r/2}^1\left | \partial_y^N \frac{\chi(y)}{r} \widetilde h \left( \frac{r}{y} \right)\right| \, dy\lesssim \frac 1{s^N}\int_{r/2}^1 \frac 1{ry^N}\, dy \lesssim  \frac 1{(sr)^N},
\]
which gives the desired bound.
\end{proof}

\section{The arithmetic function $S(q,c)$}

\subsection{An upper bound}

\begin{lemma} \label{c:trivial_sq_est}   The arithmetic function 
\[
S_{\mu L^2}(q,c) =  \sum_{\substack{a=0 \\(a,q)=1}}^{q-1}\sum_{\substack{b_i =0 \\ 1\le i \le d}}^{q-1} e \left(\frac{a Q_{\mu L^2}(b) + c\cdot b}{q}\right)
\]
is bounded by
\begin{equation}\label{eq:sqcb}
\left |S_{\mu L^2}(q,c) \right | \le C |q|^{\frac d2 +1},
\end{equation}
where the constant $C$ is independent of $\mu L^2$.
\end{lemma}
\proof Applying Cauchy Schwartz to the sum in $a$ we get
\[
\left |S_{\mu L^2}(q,c) \right |^2 \le \phi(q)  
   \sum_{\substack{a=0 \\(a,q)=1}}^{q-1}\sum_{\substack{b_i =0 \\ 1\le i \le d}}^{q-1}  \sum_{\substack{\tilde{b}_i =0 \\ 1\le i \le d}}^{q-1}
 e \left(\frac{a( Q_{\mu L^2}(b) -  Q_{\mu L^2}(\tilde{b}) )+ c\cdot( b- \tilde{b})}{q}\right)
\]
where $\phi$ is the Euler's totient function. Substituting $b= \tilde{b} + v$, we obtain using the $q-$periodicity of the summand
\begin{align*}
&\sum_{\substack{b_i =0 \\ 1\le i \le d}}^{q-1}  \sum_{\substack{\tilde{b}_i =0 \\ 1\le i \le d}}^{q-1}e \left(\frac{a( Q_{\mu L^2}(b) -  Q_{\mu L^2}(\tilde{b}) )+ c\cdot( b- \tilde{b})}{q}\right)\\
&\qquad\qquad\qquad\qquad= \sum_{\substack{v_i =0 \\ 1\le i \le d}}^{q-1}  \sum_{\substack{\tilde{b}_i =0 \\ 1\le i \le d}}^{q-1}
 e \left(\frac{a Q(v) + v\cdot c}{q}\right)e \left(\frac{a\tilde{b}\cdot \nabla Q(v)}{q}\right).
\end{align*}
 Therefore the sum in $\tilde{b}$ is zero unless $q\mid \nabla  Q(v)$, which is bounded by a constant depending on $ Q$ only, and consequently  the sum in $b$ and  $\tilde{b}$ is bounded by $O(q^d)$, which implies
 \[
 \left |S_{\mu L^2}(q,c) \right |^2 \le C \phi(q)^2 q^d,
 \] 
 and this proves the lemma.
\endproof

\subsection{Formulas for $S(q,c)$}

We considered until now general quadratic forms $Q$, but we will now focus on the specific quadratic form given by the resonance modulus of~\eqref{NLS}.

Namely, for a fixed $K \in \mathbb{R}^{n}$, and for $(K_1,K_2,K_3) \in (\mathbb{R}^{n})^3$,
$$
\Omega_3(K_1,K_2,K_3,K) = |K_1|^2 - |K_2|^2 + |K_3|^2 - |K|^2,
$$
restricted to the subset given by the condition that frequencies add up to zero
$$
K_1 - K_2 + K_3 - K = 0.
$$
This constraint implies that the quadratic form is defined on a vector space of dimension $2n$, or in other words, following the notation of the previous section,
$$
d = 2n.
$$
To take the constraint above into account, define the new coordinates $(z_1,z_2) \in \mathbb{R}^{2n}$ by
$$
\left\{
\begin{array}{l}
z_1 = K_1 - K \\
z_2 = K_3 - K,
\end{array}
\right.
$$
for which the quadratic form takes the simple form
$$
\Omega_3(K_1,K_2,K_3,K) = - 2 z_1 \cdot z_2 \overset{def}{=} \omega(z_1,z_2).
$$

Going back to the arithmetic function, we start  by noting that $S(q,c)$ can be expressed as a Ramanujan sum.  Using the explicit form of $\omega$, we have
\begin{align*}
S(q,c) &=  \sum_{\substack{a=0 \\(a,q)=1}}^{q-1}\sum_{0\leq z_i \leq q-1} e \left(\frac{a \omega(z) + c\cdot z}{q}\right)\\&=
\sum_{\substack{a=1\\(a,q)=1}}^{q-1}\prod_{m=1}^{\frac{d}{2}}\sum_{x =0}^{q-1}\sum_{y =0}^{q-1} e\left(\frac{axy+x c_m+yc_{m+\frac d2}}{q}\right).
\end{align*}
The sum in $x$ is zero unless 
 $q\mid ay+c_m$.   Then the sum in $x$ yields, 
\[
qe\left(\frac{yc_{m+\frac d2}}q\right)=qe\left(\frac{-a^*c_mc_{m+\frac d2}}q\right),
\]
where $a^*$ denotes the multiplicative inverse of $a$ modulo $q$.
Utilizing the symmetry of the sum we obtain,
\begin{equation}
S(q,c)=\sum_{\substack{a=1\\(a,q)=1}}^{q-1}\prod_{m=1}^{\frac{d}{2}}qe\left(\frac{-a^*c_mc_{m+\frac d2}}q\right)=\sum_{\substack{a=1\\(a,q)=1}}^{q-1}q^{\frac d2}e\left(\frac{a  \omega(c)}q\right)=q^{\frac d2}c_q( \omega(c))\label{e:sqc_formula}
\end{equation}
where $c_{q}(m)$ is the Ramanujan sum $c_{q}(m)$ defined by
\[
c_{q}(m):=\sum_{\substack{x=1\\(x,q)=1}}^{q-1}e\left(\frac{mx}{q}\right)~.
\]
Since Ramanujan sums are multiplicative, we immediately obtain the following lemma.
\begin{lemma}\label{l:mult_sqc}
If $(u,v)=1$ then
\[
S(uv,c)=S(u,c)S(v,c). 
\]
\end{lemma}
Moreover, since for any prime $p$ we have the explicit formula\footnote{The elementary properties of Ramanujan sums can been found in many standard number theory textbooks: see for example \cite{HW}.}
\begin{equation}\label{e:ramanujan_sum}
c_{p^j}(a)=\begin{cases}
0 & \text{if }p^{j-1}\nmid a,\\
-p^{j-1} &  \text{if } p^{j-1}\mid a \text{ and }  p^j \nmid a~,\\
 p^{j-1}(p-1) & \text{if } p^j \mid a~, \end{cases}
\end{equation}
we deduce from equations \eqref{e:sqc_formula} and \eqref{e:ramanujan_sum}  the following formula for $S(q,c)$ in case $q$ is a power of a prime.
\begin{lemma}\label{l:sp_t}
Assume $p$ is prime, then for any integer $j\geq 0$, the arithmetic function $S(p^j,c)$ may be written as
\begin{subnumcases} 
{S(p^j,c)=}
0 &  $p^{j-1}\nmid  \omega(c)$  \label{e:spt_eval-1}\\
-p^{\frac{dj}{2}+j-1} &  $p^{j-1}\mid  \omega(c) \text{ and }  p^j \nmid  \omega(c)$   \label{e:spt_eval-2}\\
 p^{\frac{dj}{2}+j-1}(p-1) & $p^j \mid  \omega(c)$  \label{e:spt_eval-3}
 \end{subnumcases} 
\end{lemma}

\subsection{Sums of $S(q,c)$ for $d \geq 6$}

We now compute certain sums of $S(q,c)$ that will appear in our asymptotic formulas for weighted lattice sums, considering first the case $d \geq 6$.

In this case we would like to compute 
\[
\sum_{q=1}^{\infty}q^{-d}S(q,0)~,
\]
which by Lemma \ref{c:trivial_sq_est} converges absolutely. Using the multiplicative nature of $S(q,c)$ (Lemma \ref{l:mult_sqc}), we have
\[
\sum_{q=1}^{\infty}q^{-d}S_q(0)=\prod_p \sum_{j=0}^{\infty}p^{-dj}S(p^j,0).
\]

From Lemma \ref{l:sp_t} it follows that for any prime $p$
\begin{equation}
 \sum_{j=0}^{\infty}p^{-dj}S(p^j,0)=1+\sum_{j=1}^{\infty}p^{-jd/2+t-1}(p-1)=\frac{p^{\frac{d}2}-1}{p^{\frac{d}2}-p}=\frac{1-p^{-\frac{d}2}}{1-p^{\frac{2-d}2}}~.
\end{equation}

Hence applying the Euler product formula for the Riemann zeta function:
\[\prod_p^{\infty}\left(1-\frac 1 {p^s}\right)=\frac{1}{\zeta(s)},\]
we obtain the following formula:
\begin{lemma}\label{weightedSqcSum}
For $d\geq 6$ we have
\[\sum_{q=1}^{\infty}q^{-d}S(q,0)=\frac{\zeta\left(\frac{d-2}{2}\right)}{\zeta\left(\frac{d}{2}\right)}~.\]
\end{lemma}

\subsection{Sums of $S(q,c)$ for $d = 4$}
If $d=4$, we distinguish two cases:
\begin{enumerate}
\item[(i)]\label{c:f0} {\it If $ \omega(c)=0$}, we remark that $S(q,c)=S(q,0)$. 
 For this case we will require asymptotic formulas on the partial sums
\begin{equation*}
M(X):=\sum_{q=1}^{X}q^{-4}S(q,0)\quad\text{and}\quad A(X):=\sum_{q=1}^{X}S(q,0)~,
\end{equation*} 
which both diverge as $X\rightarrow \infty$.
\item[(ii)] {\it If $ \omega(c)\neq 0$}, we require an asymptotic formula on
\begin{equation}\label{e:sqc_small}
A(X,c):=\sum_{q=1}^{X}S(q,c)~.
\end{equation}
\end{enumerate}

\vskip 3mm
\underline{Case (i) $ \omega(c)=0$.}  Let us consider the Dirichlet series
\[\Gamma(s):=\sum_{q=1}^{\infty}q^{-s}S(q,0)=\prod_p \sum_{j=0}^{\infty}p^{-js}S(p^j,0)~.\]
Writing $s=\sigma+it$, the sum is absolutely convergent for $\sigma> 4$. Note also the infinite sum
\[\sum_{j=0}^{\infty}p^{-js}S(p^j,0)~,\]
is convergent for $\sigma>3$. In particular we have
\[\sum_{j=0}^{\infty}p^{-js}S(p^j,0)=1+\sum_{j=1}^{\infty} p^{j(3-s)-1}(p-1)=\frac{1-p^{2-s}}{1-p^{3-s}}~.
\]
Therefore 
\begin{equation*}
\Gamma(s)=\frac{\zeta(s-3)}{\zeta(s-2)}~.
\end{equation*}

Recall $\zeta(s)$ is  bounded for $\sigma>1$.
In order to obtain asymptotic formulas for $M(X,c)$ and $A(X,c)$ we will apply well known arguments employing contour integration. Such arguments are typical in modern proofs of the Prime Number Theorem (cf.\ \cite{Apostol}).

By Perron's formula (see \cite{Apostol}, Theorem 11.18) we have for any half an odd integer $X$
\begin{align}\nonumber
M(X)&=\lim_{U\rightarrow \infty}\frac{1}{2\pi i}\int_{b-iU}^{b+iU}\Gamma(s+4)\frac{X^s}{s}~ds\\&=\lim_{U\rightarrow \infty}\frac{1}{2\pi i}\int_{b-iU}^{b+iU}\frac{\zeta(s+1)X^s}{s\zeta(s+2)}~ds~,\label{e:perron}
\end{align}
for any $b>0$.

We rewrite \eqref{e:perron} using the residue theorem
\begin{align*}
M(X)&=\lim_{U\rightarrow \infty} \frac{1}{2\pi i}\left[\int_{1-iU}^{1-iT}\dots+\int_{1-iT}^{-\frac12-iT}\dots+\int_{-\frac12-iT}^{-\frac12+iT}\dots+\int_{-\frac12+iT}^{1+iT}\dots+\int_{1+iT}^{1+iU}\dots\right]\\&\qquad+\res\left(\frac{\zeta(s+1)X^s}{s\zeta(s+2)},0\right)~,\end{align*}
where we set $T=X^{4}$. We will use throughout that on the domain of integration  $\zeta(s+2)$ is bounded from below.

To bound the integrals $\lim_{U\rightarrow \infty} \frac{1}{2\pi i}\left[\int_{1-iU}^{1-iT}\dots+\int_{1+iT}^{1+iU}\dots\right]$ we will use the bound (see \cite{Apostol} Chapter 11, Lemma 4): if $c>0$, $a \neq 1$, and $T<U$,
\[\left| \int_{c-iU}^{c-iT}\frac{a^s}{s}~ds+\int_{c+iT}^{c+iU}\frac{a^s}{s}~ds\right| \lesssim \frac{a^{c}}{T\abs{\log a}}~. \]
Then expanding $\Gamma(s+4)$ and applying the above bounds we obtain
\begin{align*}
\left| \lim_{U\rightarrow \infty} \int_{1-iU}^{1-iT} + \int_{1+iT}^{1+iU}\Gamma(s+4)\frac{X^s}{s}~ds \right| &= \left| \lim_{U\rightarrow \infty}\sum_{q=1}^{\infty}q^{-4}S(q,0)\int_{1-iU}^{1-iT} + \int_{1+iT}^{1+iU} \left( \frac{X}{q}\right)^s\frac{1}{s}~ds \right| \\
&\lesssim \sum_{q=1}^\infty q^{-5} \abs{S(q,0)} \frac{X}{ T\abs{\log X-\log q}} \\
&\lesssim \sum_{q=1}^{\infty}q^{-5} \abs{S(q,0)} \frac{X}{T\abs{\log X-\log (X+\frac 12)}}
\\&\lesssim \sum_{q=1}^{\infty}q^{-5} \abs{S(q,0)} \frac{X^{2}}{T}
\\&\lesssim X^{- 2}~.
\end{align*}

To estimate the remaining integrals  will need some well known bounds on the Riemann zeta function which we state below for the reader's convenience:
\begin{lemma}\label{l:zeta_bnd}
 f $0<\delta<1$, $1-\delta\leq \sigma\leq 2$, and $\abs{t}\geq 1$ then there exists a constant $C$ depending of $\delta$ such that
\begin{equation}\label{e:riem_bnd_sigma_large}
\abs{\zeta(\sigma+it)}\leq C\abs{t}^{\delta}~.
\end{equation}
\end{lemma}

\begin{lemma}\label{l:zeta_mean}
The zeta function satisfies the following mean estimate
\begin{equation}
\int_0^T\abs{\zeta\left(\frac12+it\right)}^2dt\lesssim T\log T~. \label{e:lindelhof}
\end{equation}
\end{lemma}

The proof of Lemma \ref{l:zeta_bnd} can be seen as a consequence of Theorem 12.23 of \cite{Apostol}. For the proof of Lemma \ref{l:zeta_mean} see \cite{Titchmarsh}, Theorem 7.2A.

Applying \eqref{e:riem_bnd_sigma_large} with $\delta=\frac12$ we obtain
\begin{align*}
\abs{\int_{1-iT}^{-\frac12+iT}\dots+\int_{-\frac12-iT}^{1+iT}\dots}
&\lesssim T^{\frac12} \left[\int_{1-iT}^{-\frac12-iT}\abs{ X^ss^{-1}}~ds+\int_{\frac12-iT}^{1+iT}\abs{ X^ss^{-1}}~ds\right]
\\&\lesssim T^{-\frac12} X
\\&\lesssim X^{-1}.
\end{align*}

From \eqref{e:lindelhof} we trivially have
\[\int_{e^j}^{e^{j+1}}\abs{\zeta\left(\frac12+it\right)}^2~dt \lesssim (j+1)e^{j+1}\]~,
and hence applying H\"older's inequality we obtain
\begin{align*}
\int_{0}^{T}\abs{\frac{\zeta\left(\frac12+it\right)}{\frac{1}{2}+it}}~ds
&\leq \int_{0}^{1}\abs{\frac{\zeta\left(\frac12+it\right)}{\frac{1}{2}+it}}~ds+\sum_{j=0}^{\log T} \int_{e^j}^{e^{j+1}} \abs{\frac{\zeta\left(\frac12+it\right)}{\frac{1}{2}+it}}~ds\\
&\lesssim 
1+\sum_{j=0}^{\log T} \left[\int_{-\frac12+ie^j}^{-\frac12+ie^{j+1}} \abs{\zeta\left(\frac12+it\right)}^2~ds\right]^{\frac12}e^{-\frac{j}{2}}
\\&\lesssim 
1+\sum_{j=0}^{\log T} j^{\frac12}\\&\lesssim \left(\log T\right)^{\frac 32}~.
\end{align*}

Thus
\begin{align*}
\abs{\int_{-\frac12-iT}^{-\frac12+iT}\dots}\lesssim X^{-\frac 12}\left(\log X\right)^{\frac 32}~.
\end{align*}

Finally from the residue formula for second order poles we have
\begin{align*}
\res\left(\frac{\zeta(s+1)X^s}{\zeta(s+2)s},0\right)&=\lim_{s\rightarrow 0} \frac{d}{ds}\left[s^2 \frac{\zeta(s+1)X^s}{s\zeta(s+2)}\right] 
\\&=\log(X)\lim_{s\rightarrow 0} \left[\frac{s \zeta(s+1)}{\zeta(s+2)}\right]\\
&+\lim_{s\rightarrow 0} \frac{d}{ds}\left[\frac{s \zeta(s+1)}{\zeta(s+2)}\right]
\\&=\frac{\log(X)}{\zeta(2)}+\lim_{s\rightarrow 0}\frac{\zeta(s+1)+s\zeta'(s+1)}{\zeta(s+2)}-\lim_{s\rightarrow 0}\frac{s\zeta(s+1)\zeta'(s+2)}{\zeta(s+2)^2} \\
&=\frac{\log(X)+\gamma}{\zeta(2)}-\frac{\zeta'(2)}{\zeta(2)^2}~,
\end{align*}
where $\gamma$ is the Euler-Mascheroni constant.

Now consider the sum $A(X)$.
By Perron's formula we have for any half an odd integer $X$
\begin{align}\nonumber
A(X)&=\lim_{T\rightarrow \infty}\frac{1}{2\pi i}\int_{b-iT}^{b+iT}\Gamma(s)\frac{X^s}{s}~ds~,
\end{align}
for any $b>4$.

Then 
\begin{align*}
A(X)&=\lim_{U\rightarrow \infty} \frac{1}{2\pi i}\left[\int_{5-iU}^{5-iT}\dots+\int_{5-iT}^{\frac72-iT}\dots+\int_{\frac72-iT}^{\frac72+iT}\dots+\int_{\frac72+iT}^{5+iT}\dots+\int_{5+iT}^{5+iU}\dots\right]\\&\qquad+\res\left(\Gamma(s)\frac{X^s}{s},4\right)~.\end{align*}
Set $T=X^{12}$. Applying identical arguments as those used in the asymptotic formula of $M(X)$ we obtain an error
of size $X^{\frac72}\left(\log T\right)^{\frac 32}$ resulting from the integrals above.

The residue of the simple pole at $s=4$ is then
\begin{equation*}
\res\left(\Gamma(s)\frac{X^s}{s},4\right)=\lim_{s\rightarrow 4}(s-4)\frac{\zeta(s-3)X^s}{s\zeta(s-2)}=\frac{X^4}{4\zeta(2)}
\end{equation*}

Collecting the above computations, we obtain:
\begin{lemma}\label{l:SLC_Omega_0}
Suppose $d=4$ and $ \omega(c)=0$ then
\begin{equation}\label{eq:lnb}
\sum_{q=1}^{X}q^{-4}S(q,c)=M(X)=\frac{\log(X)+\gamma }{\zeta(2)}-\frac{\zeta'(2)}{\zeta(2)^2}+O(X^{-\frac 12}\left(\log X\right)^{\frac 32})~,
\end{equation}
and
\begin{equation}\label{eq:Aqc}
\sum_{q=1}^{X}S(q,c)=A(X)=\frac{X^4}{4\zeta(2)}+O(X^{\frac72}\left(\log X\right)^{\frac 32})~.
\end{equation}
\end{lemma}

\vskip 3mm
\underline{Case (ii) $ \omega(c)\ne 0$.} To obtain a bound on \eqref{e:sqc_small} we start by considering the function
\begin{equation*}
\Gamma(s,c):= \sum_{q=1}^{\infty} q^{-s}S(q,c)= \prod_p \sum_{j=0}^{\infty}p^{-js}S(p^j,c)~.
\end{equation*}
From Lemma \ref{l:sp_t},  if $p\nmid  \omega(c)$ we have from \eqref{e:spt_eval-1} and \eqref{e:spt_eval-2}
\begin{equation*}%\label{eq:spc1}
\sum_{j=0}^{\infty}p^{-js}S(p^j,c)=1-p^{2-s}~.
\end{equation*}
and hence
\[
\prod_{p\nmid  \omega(c)} \frac{1}{1-p^{2-s}}\sum_{j=0}^{\infty}p^{-js}S(p^j,c)=1~.\]
If $p\mid  \omega(c)$ and $s> 3$ we have from Lemma \ref{l:sp_t}
\begin{equation*}%\label{eq:spc2}
\sum_{j=0}^{\infty}p^{-js}S(p^j,c)=1+O(p^{3-s}).
\end{equation*}
Since $ \omega(c)$ has at most $O\left(\frac{\log\abs{c}}{\log\log\abs{c}}\right)$ distinct prime divisors\footnote{This is a consequence of the Prime Number Theorem, see \cite{HW}, Section 22.10.}  we have for $s>3$
\[
\prod_{p\mid  \omega(c)} \frac{1}{1-p^{2-s}}\sum_{j=0}^{\infty}p^{-js}S(p^j,c)\leq 
\prod_{p\mid  \omega(c)} \sum_{j=0}^{\infty}p^{-js}S(p^j,c)\leq e^{O\left(\frac{\log\abs{c}}{\log\log\abs{c}}\right)}
\lesssim c^{\epsilon}~.
\]
Thus by Euler's product formula we obtain
\begin{equation}\label{eq:spc}
\Gamma(s,c):=\prod_p \sum_{j=0}^{\infty}p^{-js}S(p^j,c)=\nu(c,s)\zeta(s-2)
\end{equation}
for some  function $\nu(c,s)$, analytic of  order  $O(c^{\epsilon})$ for $\Re s>3$. In particular we obtain
\[\sum_{q=1}^{X}\frac{S(q,c)}{q^{s}}\lesssim c^{\epsilon},\]
for any $s>3$. Using Abel's summation formula \eqref{e:Abel_summation} in conjunction with \eqref{eq:spc} we obtain that for $s>3$
\begin{align*}
\sum_{q=1}^{X}S(q,c)&= X^{s}\sum_{q=1}^{X}\frac{S(q,c)}{q^{s}}-s\int_1^{X} u^{s-1}\sum_{q=1}^{u}\frac{S(q,c)}{q^{s}}~du\\
& \lesssim c^{\epsilon}X^{s}~.
\end{align*}

Hence we obtain the following lemma:
\begin{lemma}\label{l:SLC_Omega_not_0}
Suppose  $d=4$ and $ \omega(c)\neq 0$ then for any $\epsilon>0$ we have
\begin{equation}
\sum_{q=1}^{X}S(q,c)\lesssim  c^{\epsilon}X^{3+\epsilon}~.
\end{equation}
\end{lemma}

%July 1

%July 1

%may26
\section{Sharp upper bounds on  lattice sums}\label{sec:res-bound}

The main purpose of this section is to prove the following,
\begin{theorem}\label{th:disbnd}
Let $K\in \mathbb{Z}^n_L$, and 
\[
\mathcal R_\mu(K) = \{K_i \in \mathbb{Z}_L^n; \;  \mathcal{S}_3(K) =  K_1 - K_2 + K_3 - K = 0, \;  \Omega_3(K) = K_1^2 - K_2^2 + K_3^2 - K^2 =\mu\}.
\]  
Given sequences $\{a_K\}$, $\{b_K\}$, and $\{c_K\}$,  such that 
\[
\left |a_K\right | + \left |b_K\right | + \left |c_K\right | \lesssim \langle K \rangle^{-\ell}
\]
we have for $\ell > 3n+2$
\begin{equation}
\sup_{K,\mu}\;  \langle K \rangle^{\ell}\sum_{\mathcal R_\mu(K)} a_{K_1}b_{K_2}c_{K_3} \lesssim  
\begin{cases} L^{2n-2}  \quad \text{if} \quad n>2~,\\											
 L^2 \log L  \quad \text{if} \quad n=2~.
\end{cases}
\label{eq:rsnntbnd5}
\end{equation}
\end{theorem}

Note that although it suffices to show that \eqref{eq:rsnntbnd5} holds with $a_K, b_K, c_K$ replaced by $\langle K \rangle^{-\ell}$,
\begin{equation}\label{eq:rsnntbnd7}
\sup_{K,\mu}\;  \langle K \rangle^{\ell}\sum_{\mathcal R_\mu(K)}\langle K_1 \rangle^{-\ell}\langle K_2 \rangle^{-\ell}\langle K_3 \rangle^{-\ell} \lesssim \begin{cases} L^{2n-2}  \quad \text{if} \quad n>2~,\\		 L^2 \log L  \quad \text{if} \quad n=2~.
\end{cases}
\end{equation}
We will prove the above bound where each $\langle K_i \rangle^{-\ell}$ is replaced by $f_i(K_i)$, with $f_i\in X^{\ell,N}$. Indeed, some of the intermediate steps with these more general functions will turn out to be useful in the next section.

As explained in the previous section, we switch variables to
$$
\left\{
\begin{array}{l}
z_1 = K_1 - K \\
z_2 = K_3 - K
\end{array}
\right.
$$
so that
$$
\Omega_3(K,K_1,K_2,K_3) = -2 z_1 \cdot z_2 = \omega(z)~.
$$

Inequality \eqref{eq:rsnntbnd7} will be proved by applying  Theorem \ref{th:cm} to the weight
\begin{equation}
\label{eq:wz1}
W(z) =  f_1(K+z_1)f_2(K + z_1 + z_2)f_3( K + z_2)~,
\end{equation}
which gives
\begin{equation}
\label{eq:rsnntbnd8}
\sum_{\mathcal R_\mu(K)} f_1(K_1)f_2(K_2) f_3(K_3) =\sum_{\substack{z\in \mathbb{Z}_L^{2n}\\ \omega_\mu(z) =0}} W(z)=
\frac{L^{2n-2}}{C_L}\sum_{q=1}^L \sum_c S_{\mu L^2}(q,c) \frac 1{q^d} I_\mu(r, c)~,
\end{equation}
where $\omega_{\mu}(z)=\omega(z)-\mu$. Notice that  $W$ (and hence $I_\mu(r,c)$) implicitly depends on $K$.  Here, we restricted the sum in $q$ to $q\leq L$ using the knowledge that $I_\mu(r, c)$ is supported on $r=\frac{q}{L}\in (0,1)$.

The decay of the sum in $K$ is due to the restriction
 $K_1 - K_2 + K_3 = K$, which implies  that $|K_i| \geq {|K|}/{3}$
for at least one $i$. Consequently one of the $f_i$ in the sum will contribute $\langle K \rangle^{-\ell}$ 
and the remaining two $f_i$ will be used in  bounding the sum. Therefore we introduce cutoff functions supported on $|K_i| \gtrsim |K|$, for $i=1,2,\mbox{ or } 3$.

   Let $\varphi \in C_0^\infty(-2,2)$ with $\varphi(x) = 1$ for $|x| < 1$. With
a slight abuse of notation, we write
\begin{align*}
\chi_{\{|x| < a\}} &= \varphi\bigg(\frac{|x|}{a}\bigg)\\
\chi_{\{|x| > a\}} &= 1 -\varphi\bigg(\frac{|x|}{a}\bigg)~,
\end{align*}
and introduce the following cutoff functions:
\begin{align*}
\chi_0 &= \chi_{\{|K| < 1\}}\\
\chi_1 &= \chi_{\{|K| > 1\}}\chi_{\{|K_1| > \frac{K}{8}\}}\\
\chi_2 &= \chi_{\{|K| > 1\}}\chi_{\{|K_1| <  \frac{K}{8}\}}\chi_{\{|K_2| > \frac{K}{8}\}}\\
\chi_3 &= \chi_{\{|K| > 1\}}\chi_{\{|K_1| <
  \frac{K}{8}\}}\chi_{\{|K_2| < \frac{K}{8}\}} \chi_{\{|K_3| >
  \frac{K}{8}\}}.
\end{align*}
%june9
Since $\chi_{\{|K_1| < \frac{|K|}{8}\}}\chi_{\{|K_2| < \frac{K}{8}\}}\chi_{\{|K_3| <  \frac{K}{8}\}} = 0$ 
we have 
\[
\chi_0 + \chi_1 + \chi_2 + \chi_3 = 1.
\]
Writing  $W_i = W\chi_i$,
\begin{equation}
\label{eq:wi}
\begin{split}
&|W_1(z)| \lesssim   \langle K \rangle^{-\ell} f_2(K + z_1 + z_2)f_3( K + z_2), \\[.3em]
&|W_2(z)| \lesssim   \langle K \rangle^{-\ell} f_1(K + z_1)f_3( K + z_2), \\[.3em]
&|W_3(z)| \lesssim   \langle K \rangle^{-\ell}f_1( K + z_1) f_2(K + z_1 + z_2).
\end{split}
\end{equation}
Then $I_\mu(r,c)$ can be written as 
\begin{equation}
\begin{split}
\label{penguin}
I_\mu(r,c) &= \int_{\mathbb R^{2n}} W(z)  \int_{-\infty}^{\infty} \widehat h_\chi\left(r,s\right) e(s \omega_{\mu}) e\left(-\frac{c\cdot z}{r}\right)\, dsdz,\\
&= \int_{-\infty}^{\infty}  \widehat h_\chi\left(r,s\right) e\left(-s\mu\right)\int_{\mathbb R^{2n}} W(z)  e(s \omega ) e\left(-\frac{c\cdot z}{r}\right)\, dzds,\\&=
\int_{-\infty}^{\infty} \widehat{h}_\chi(r,s) e\left(-s\mu\right) v\left(s, -\frac cr\right) \, ds\\
\end{split}
\end{equation}
where $v$ solves the following Schr\"odinger-type equation
\begin{equation}
\label{eq:hypsch2}
\begin{split}
&\partial_s \widehat{v}(s,\xi) = 2\pi i \omega(\xi)  \widehat{v}(s,\xi) \\
& \widehat{v}(0,\xi) = W(\xi).
\end{split}
\end{equation}

Writing $I_\mu(r,c) =\sum_{j=0}^3 I_{j,\mu}(r,c)$ we have
\begin{equation*}
\begin{split}
 I_{j,\mu}(r,c) &= \int_{\mathbb R^{2n}} W_j(z)  h_\chi\left(r,\omega_{\mu}(z)\right)  e\left(-\frac{c\cdot z}{r}\right)\, dz,\\ &= \int_{-\infty}^{\infty} \widehat{h}_\chi(r,s) e\left(-s\mu\right) v_j\left(s, -\frac cr\right) \, ds,
\end{split}
\end{equation*}
where $v_j = v \chi_j$.

The $v_j$, as solutions of  this dispersive PDE, satisfy the following elementary bound.

\medskip 

\begin{lemma}\label{lem:disp}
If $u$ solves the equation 
$$
\partial_s \widehat u(s, z)= 2\pi i \omega(z)  \widehat u(s,z),
$$
then
\begin{equation}\label{eq:elemest}
\begin{split}
\|u(s)\|_{L^\infty(\RR^{2n})}\lesssim& \|\widehat u(0)\|_{L^1(\RR^{2n})}\\
\|u(s)\|_{L^\infty(\RR^{2n})} \lesssim& \frac{1}{s^{n}}\|(1-\Delta)^m \widehat u(0)\|_{L^2(\RR^{2n})}
\end{split}
\end{equation}
for any integer $m>\frac{n}{2}$.   Consequently, for $v$ defined by ~\eqref{eq:hypsch2}, with $f \in X^{\ell,N}$ with $\ell > \frac{n}{2}$ and $N>n$,
\begin{equation}\label{PrecisedvEst-0}
\abs{v\left(s, -\frac{c}{r}\right)} \le \sum_{j=0}^3 \abs{v_j\left(s, -\frac{c}{r}\right)}\lesssim \langle s \rangle^{-n}\langle K \rangle^{-\ell}.
\end{equation}
\end{lemma}

\begin{proof}
The first inequality is trivial and follows from the the fact $\widehat u(s)=e(\omega s) \widehat u(0)$. For the second inequality, we use the dispersive estimate
\begin{align*}
\|u(s)\|_{L^\infty} \lesssim s^{-n} \|u(0)\|_{L^1}\lesssim s^{-n} \|\langle x \rangle^{2m} v\|_{L^2}\sim s^{-n}\|(1-\Delta)^m \widehat v\|_{L^2}.
\end{align*}
Using inequalities \eqref{eq:wi} finishes the proof.
\end{proof}

The above lemma gives bounds on $v_j$  that are independent of $r$ and $c$, and therefore  
\[
\left |  I_\mu(r,c) \right | \lesssim1.
\]
To obtain bounds that decay for large $c$ or for small $r$, we need to integrate by parts in $z$ in the expression
\[
v_j\left(s,-\frac cr\right) = \int_{\mathbb R^{2n}} W_j(z) e{\left(s\omega(z) -\frac{c\cdot z}{r}\right)}\, dz.
\]
However, to avoid introducing  powers of $K$ from derivatives of $\omega(z)$ (since $W$ is a function  of $K+z_i$), one should exercise care when doing so.  
\medskip 

\begin{proposition}\label{dispersive lemma}
Let $W$ be  given by \eqref{eq:wz1} with $f_j\in X^{\ell, N}$ for $\ell>3n+2$ and $N> 3n+2$ with $\|f_j\|_{\ell, N} = 1$ for $j=1, 2, 3$, then
\begin{equation}
\abs{v\left(s, -\frac{c}{r}\right)}\lesssim  \;\langle K \rangle^{-\ell} \frac{ \langle rs \rangle^{2n+2} }{\langle s \rangle^{n}} (A_0 + A_1 +A_2 +A_3)\label{asfoor2}
\end{equation}
 where
 \begin{align*}
&A _0= \langle c_1\rangle^{-(n+1)}\langle c_2\rangle^{-(n+1)}\\
&A_1 =  \langle c_1 -2Krs\rangle^{-(n+1)}\langle c_2\rangle^{-(n+1)}\\
&A_2 =  \langle c_1 -2Krs\rangle^{-(n+1)}\langle c_2-2Krs\rangle^{-(n+1)}  \\
&A_4 = \langle c_1\rangle^{-(n+1)}\langle c_2-2Krs\rangle^{-(n+1)}.
\end{align*}
Furthermore, if $c$ is distinct from $(0,0)$, $([2Krs],0)$, $(0,[2Krs])$, $([2Krs],[2Krs])$
$$
\abs{v\left(s, -\frac{c}{r}\right)}\lesssim  \;\langle K \rangle^{-\ell} r^{n+1} \langle s \rangle \langle rs \rangle^{n+1} (A_0 + A_1 +A_2 +A_3).
$$
\end{proposition}

%\comment{We only need to take $m_1=m_2= n+1$.}
\begin{proof}
We will estimate each $v_j$ by splitting $\omega(z)$ in an appropriate manner and integrating by parts in $z$.

\underline{Estimate on $v_0$.}  Since $W_0=0$ for $|K| > 2$, the presence of $K$ in $W_0$ plays no role.  In this case we can obtain decay for large $c$ or small  $r$ by directly integrating by parts  $e\left(-\frac{c\cdot z}r\right)$ in the expression
\begin{equation}\label{eq:v0}
v_0\left(s, -\frac{c}{r}\right)=\int W_0(z) e(s\omega(z)) e\left(-\frac{c \cdot z}{r}\right) dz.
\end{equation}
Assume first that $c_1 \neq 0$, then one of the $n$ components $c_{1, j}$ of $c_1$ satisfies $|c_{1,j}| \geq n^{-1}|c_1|$.  Similarly, if  $c_2 \neq 0$, then one of the $n$ components $c_{2, k}$ of $c_2$ satisfies $|c_{2,k}| \geq n^{-1}|c_2|$.  For $c_i\ne 0$ for $i = 1,2$, using the identity 
 \begin{equation*}
 e{\left( -\frac{c\cdot z}{r}\right)}
=\left(\frac{r}{(-2\pi i)}\right)^{m_1+m_2} \!\!(c_{1, j})^{-m_1}(c_{2,k})^{-m_2}
\left(\frac{\partial}{\partial z_{1,j}}\right)^{m_1}\left(\frac{\partial}{\partial z_{2, k}}\right)^{m_2}  e{\left( -\frac{c\cdot z}{r}\right)},
\end{equation*}
we integrate by parts in \eqref{eq:v0} $m_1=m_2= n+1$ times.  Observe the bound 
\begin{align*} 
&r^{m_1+m_2}\left| \left(\frac{\partial}{\partial z_{1,j}}\right)^{m_1}\left(\frac{\partial}{\partial z_{2, k}}\right)^{m_2} W_0(z) e(\omega(sz))\right|  \\
&\lesssim r^{m_1+m_2}\sum_{\substack{\alpha+\beta=(m_1, m_2)\\ \alpha, \beta \in \mathbb{N}^2}} \left|\left(\frac{\partial}{\partial z_{1,j}}\right)^{\alpha_1}\left(\frac{\partial}{\partial z_{2, k}}\right)^{\alpha_2} W_0(z)\right| (2\pi i s)^{|\beta|} |z_2|^{\beta_1} |z_1|^{\beta_2} \nonumber \\
&\lesssim \langle r s \rangle^{m_1+m_2} \langle K \rangle^{-\ell}  \langle K_1 \rangle^{-\ell}  \langle K_2 \rangle^{-\ell}\langle K_3  \rangle^{-\ell}\langle K_3- K\rangle^{m_1} \langle K_1-K\rangle^{m_2}\chi_0.
\end{align*}

Then after integrating by parts one can write \eqref{eq:v0} as a sum of terms of the type
\[P(r,s)\int V(z) e(s\omega(z))e\left(-\frac{c \cdot z}{r}\right) dz=P(r,s)v\left(s-\frac cr\right)\]
where $\abs{P(r,s)}\lesssim \langle r s \rangle^{m_1+m_2} $ and $\widehat v$ solves the Schr\"odinger-type equation \eqref{eq:hypsch2} with initial data $V$. Additional derivatives falling on $V$ correspond to additional derivatives falling on $W_0$ and hence by estimate \eqref{eq:elemest} (distinguishing between the cases $s<1$ and $s>1$), we obtain the contribution $A_0$.  If either $c_1=0$ or $c_2=0$, we simply integrate by parts $n+1$ times on the non zero component of $e\left(-\frac{c\cdot z}r\right)$. If $c_1 = c_2 = 0$, estimate~\eqref{PrecisedvEst-0} is the desired bound.

\medskip

\underline{Estimate on $v_1$.}  Since in this case $|K_1| \gtrsim |K|$, we want to ensure that any integration by parts does not yield a power of $K_1$.  For this reason we write
\begin{align*}
\omega(z)&= |K + K_2 - K_3|^2 - |K_2|^2 + |K_3|^2 - |K|^2\\
\phantom{\omega(z)} & = 2|K_3|^2 - 2K_2\cdot K_3 + 2K\cdot (K_2-K_3)\\
\phantom{\omega(z)} & = 2K_3\cdot(K_3 - K_2) + 2K\cdot(K_1-K)\\
\phantom{\omega(z)} &\overset{def}{=} \omega_1(z) + 2K\cdot z_1
\end{align*}
Since $ \omega_1(z)= 2K_3\cdot(K_3 - K_2) $, any derivative of $\omega_1(z)$ can be bounded by a polynomial expression in $\langle K_3\rangle$ and $\langle K_3 -K_2\rangle$.   Therefore, writing 
\begin{multline*}
v_1\left(s, -\frac{c}{r}\right)=\int W_1(z) e(s\omega(z)) e\left(-\frac{c \cdot z}{r}\right) dz=\\
\int W_1(z) e(s\omega_1(z)) e\left(-\frac{(c_1-2Krs) \cdot z_1+c_2 \cdot z_2}r\right) dz,
\end{multline*}
for $|c_1- 2Krs| > \frac{1}{2}$ or $c_2\ne 0$, we  integrate by parts $m_1=n+1$ in $z_1$ and $m_2=n+1$ in $z_2$.  Using the identity 
\begin{multline*}    
 e(-r^{-1}\left(\left(c_1-2Krs\right) \cdot z_1+c_2\cdot z_2\right) 
=\left(\frac{r}{(-2\pi i)}\right)^{m_1+m_2} (c_{1, j}-2K_jrs)^{-m_1}(c_{2,k})^{-m_2}\times\\
\left(\frac{\partial}{\partial z_{1,j}}\right)^{m_1}\left(\frac{\partial}{\partial z_{2, k}}\right)^{m_2} e(-r^{-1}\left(\left(c_1-2Krs\right) \cdot z_1+c_2\cdot z_2\right) ,
\end{multline*}
the bound
\begin{align*} 
&r^{m_1+m_2}\left| \left(\frac{\partial}{\partial z_{1,j}}\right)^{m_1}\left(\frac{\partial}{\partial z_{2, k}}\right)^{m_2} W_1(z) e(s\omega_1(z))\right|  \\
&\lesssim r^{m_1+m_2}\sum_{\substack{\alpha+\beta=(m_1, m_2)\\ \alpha, \beta \in \mathbb{N}^2}} \left|\left(\frac{\partial}{\partial z_{1,j}}\right)^{\alpha_1}\left(\frac{\partial}{\partial z_{2, k}}\right)^{\alpha_2} W_1(z)\right| (2\pi i s)^{|\beta|} |\partial_{z_1} \omega_1|^{\beta_1} |\partial_{z_1} \omega_1|^{\beta_2} \nonumber \\
&\lesssim \langle rs \rangle^{m_1+m_2} \langle K \rangle^{-\ell}  \langle K_2 \rangle^{-\ell}\langle K_3 \rangle^{-\ell} ( \langle K_3 \rangle + \langle K_2 \rangle)^{m_1 + m_2},
\end{align*}
along with analogous bounds on derivatives, and finally the estimate \eqref{eq:elemest}, we obtain the contribution $A_1$.

If either $|c_1- 2Krs|<\frac 12$ or $c_2=0$ we proceed as above and only integrate by parts on the nontrivial components.

\medskip

\underline{Estimate on $v_2$.}  Since in this case $|K_2| \gtrsim |K|$, we want to ensure that any integration by parts does not yield a power of $K_2$.  For this reason we write
\begin{align*}
\omega(z)&= |K_1|^2 - |K_1 + K_3 - K|^2 + |K_3|^2 - |K|^2\\
\phantom{\omega(z)} &= -2K_1\cdot K_3 + 2K\cdot(K_1+K_3)-2|K|^2\\
\phantom{\omega(z)} &= 2K_1\cdot K_3 + 2|K|^2 +  2K\cdot(z_1 +z_2)\\
\phantom{\omega(z)} &\overset{def}{=} \omega_2 + 2K\cdot(z_1 + z_2)
\end{align*}
Since $ \omega_2(z)=2K_1\cdot K_3 + 2|K|^2 $, any derivative of $\omega_2(z)$ can be bounded by a polynomial expressions in $\langle K_1\rangle$ and $\langle K_3 \rangle$.   Therefore we write 
\begin{multline*}
v_2\left(s, -\frac{c}{r}\right)=\int W_2(z) e(s\omega(z)) e\left(-\frac{c \cdot z}{r}\right) dz=\\
\int W_2(z) e(s\omega_2(z)) e\left(-\frac{(c_1-2Krs) \cdot z_1+\left(c_2 -2Krs\right)\cdot z_2}r\right)dz,
\end{multline*}
and proceed exactly as in the case of $v_1$, which ultimately gives the contribution $A_2$.

\medskip

\underline{Estimate on $v_3$.}  It is symmetrical to $v_1$ and will therefore not be detailed.
\end{proof}
 We  finish the proof of \eqref{eq:rsnntbnd7} (and hence that of Theorem \ref{th:disbnd}) by 
 using \eqref{eq:sqcb} to bound
 \[
 \sup_c \left |S_{\mu L^2}(q,c) \right | \lesssim |q|^{n +1},
\]
and combining equation \eqref{penguin}, Proposition \ref{dispersive lemma},  and  estimate \eqref{eq:hhatb} to bound
\begin{equation}\label{eq:ircbnd}
\begin{split}
\sum_c | I_\mu(r,c)| \lesssim& \langle K \rangle^{-\ell}\int_{\RR} |\widehat h_{\chi}(r,s)|\frac{\langle rs\rangle^{2n+2}}{\langle s \rangle^{n}}   \sum_c\sum_{i=0}^3 A_i\,ds\\
  \lesssim& \langle K \rangle^{-\ell} \int_{\RR}  \frac{1}{\langle s \rangle^{n}} \,ds
\lesssim \langle K \rangle^{-\ell}.
\end{split}
\end{equation}
Thus from  \eqref{eq:rsnntbnd8} we conclude,
\begin{align*}
\sum_{\substack{z\in \mathbb{Z}_L^{2n}\\ \omega_\mu(z) =0}}  W(z)\lesssim&
 \frac{L^{2n-2}}{C_L}\sum_{q=1}^L \sum_c |S_{\mu L^2} (q,c)| \frac 1{q^{2n}} | I_\mu(r,c)|\\
 \lesssim&\langle K \rangle^{-\ell}L^{2n-2}\sum_{q=1}^L q^{-n+1} \lesssim \langle K \rangle^{-\ell}\begin{cases} L^{2n-2}  \quad \text{if} \quad n>2,\\											
 L^2 \log L  \quad \text{if} \quad n=2.
\end{cases}
\end{align*}

%%%%june9

\section{Asymptotics of the resonant sum}\label{sec:asymp}

The purpose of this section is to prove the following theorem:

\begin{theorem}\label{th:asymptotics}
Let $f_1, f_2, f_3\in X^{\ell+n+2, N}(\RR^n)$ for $\ell>2n$ and $N>3n+2$  and  set 
\begin{equation}\label{eq:wz}
W(z) = f_1( K + z_1) \bar{f}_2(K +z_1 +
z_2)f_3( K + z_2),
\end{equation}
and
\begin{equation}\label{def of T(W)}
\mathcal T(W)(K)=\mathcal T(f_1, f_2, f_3)(K):=\int_{\mathbb{R}^{2n}} \delta(\omega(z)) W(z) dz,
\end{equation} 
and recall that $\omega(z)=z_1 \cdot z_2$ and
\[
Z_n(L) = \left\{
\begin{array}{ll}
\frac{1}{\zeta(2)} L^2 \log L & \mbox{if $n=2$} \\
\frac{\zeta(n-1)}{\zeta(n)} L^{2n-2} & \mbox{if $n \geq 3$}
\end{array} \right.
\]
\begin{enumerate}
\item[1)] For $n\geq 3$, define 
\begin{align*}
&\Delta(W) =\frac{1}{Z_n(L)}\sum_{\substack{z\in \mathbb{Z}_L^{2n}\\ \omega(z) =0}} W(z) - \mathcal T(W),
\end{align*}
%If $f_j\in X^{\ell+n+2, N}(\RR^n)$ for $\ell>2n$ and $N>3n+2$ then
\begin{equation}
\|\Delta(W)\|_{X^\ell} \lesssim \prod_{i=1}^3 \|f_i\|_{X^{\ell+n+2, N}(\RR^n)}\begin{cases} L^{-1}\log L \quad \text{ if } n=3\\
L^{-1} \quad \text{ if }n \geq 3
\end{cases}
\label{eq: Delta bound}
\end{equation}

\item[2)] For $n=2$, define
\begin{align}
\widetilde \Delta(W) =\frac{1}{Z_2(L) }  \sum_{K_1 \cdot K_3 = 0}W(K_1, K_3) - \left(\mathcal T(W)  + \frac{\zeta(2)}{\log L}\mathcal C(W)\right) \label{eq: def of tilde Delta}
\end{align}
where $\mathcal C(W)$ is a correction operator that is independent of $L$ and is defined explicitly in \eqref{def of C(K)}.

If $f_j\in X^{\ell+6, N}(\RR^n)$ for $\ell>4$ and $N>14$ then
\begin{equation}\label{eq: Delta bound2}
\begin{split}
&\| \mathcal C(W)\|_{X^\ell} \lesssim \prod_{i=1}^3 \|f_i\|_{X^{\ell+6, N}(\RR^2)}.\\
&\|\widetilde \Delta(W)\|_{X^\ell} \lesssim L^{-1/3+}\prod_{i=1}^3 \|f_i\|_{X^{\ell+6, N}(\RR^2)} 
\end{split}
\end{equation}
\end{enumerate}
\end{theorem}

This  theorem  is proved by finding the asymptotics of the resonant sum using the circle method: recall that Theorem \ref{th:cm} gives
\[
\sum_{\substack{z\in \mathbb{Z}_L^{2n}\\ \omega(z) =0}} W(z)= \frac{L^{2n-2}}{C_L}
\sum_{q=1}^L \sum_c S(q,c) \frac 1{q^d} I(r, c)~.
\]
Thus the proof of the  theorem amounts to finding the asymptotics to $I(r, c)$  and $S(q,c)$.

\subsection{Analysis of $I(r, 0)$}
Recall that 
\[
I(r,0) = \int_{\mathbb R^{2n}} W_\chi(z)h_\chi\left(r,\omega(z)\right) \, dz
= \int_{\mathbb R} h_\chi(r, \rho) \mathcal I(\rho)\, d\rho
\]
where 
\begin{equation}\label{def of Irho}
 \mathcal I(\rho) =\chi(\rho)\int_{\mathbb R^{2n}} \delta(\omega(z)- \rho)W(z) dz~.
\end{equation}

$\mathcal I(\rho)$ can be  written as
\begin{equation*}
 \mathcal I(\rho) =\chi(\rho)\int_{\RR^{2n}} \int_{\RR}e(s\Omega_3(z)-s\rho) ds W(z) \,dz=\chi(\rho)\int_{\RR}e(-s\rho) v(s, 0)\, ds~.
\end{equation*}
where $v$ is a solution of the PDE
\[
\partial_s \widehat{v}(s,\xi) = 2\pi i \Omega_3(\xi)  \widehat{v}(s,\xi),  \qquad  \widehat{v}(0,\xi) = W(\xi)~,
\]
and from \eqref{PrecisedvEst-0}, we have
\begin{equation}\label{aim0}
\sup_{|\rho|\le 1/2} |\mathcal{I}(\rho)| \lesssim \langle K\rangle^{-\ell}\prod_{i=1}^3 \|f_i\|_{X^{\ell}}~,
\end{equation}
\begin{equation}\label{eq:boundg}
|I(r,0)|   \lesssim \langle K\rangle^{-\ell}\prod_{i=1}^3 \|f_i\|_{X^{\ell}}~.
\end{equation}

\begin{lemma}\label{lem:boundI}
Suppose $f_i \in X^{\ell, 1}$ for $\ell >3n+2$ and $i=1, 2, 3$. For $n\ge 3$, $\mathcal{I}\in C^1(-\frac 12,\frac 12)$.  For $n=2$,  $\mathcal I(\rho)$  in Lipschitz on $(-1/2, 1/2)$, and therefore for $n\ge 2$
\begin{equation}\label{eq:limitI}
\left |I(r,0) - \mathcal{I}(0)\right | \lesssim  r \langle K\rangle^{-\ell}\prod_{i=1}^3 \|f_i\|_{X^{\ell, 1}}~.
\end{equation}
\end{lemma}

\proof  
Without loss of generality we can assume  $\|f_j\|_{\ell, 1} = 1$ for $j=1, 2, 3$.

For $n\ge 3$, the result follows from writing  $\mathcal I(\rho) =\chi(\rho)\int_{\RR}e(-s\rho) v(s, 0) ds$  and the fact that $ |v(s,0) \lesssim \langle s \rangle^{-n}\langle K \rangle^{-\ell}$.

For $n=2$, we  introduce coordinates adapted to the surface
$\omega(z) = \rho$. This is accomplished as follows: First we
rotate our coordinates $(z_1, z_2) \overset{\Phi}{\to} (x,y) \in
\mathbb{R}^{4}$
\begin{align*}
z_1 &= y + x\\
z_2 &= y-x~.
\end{align*}
and then use $\rho = |y|^2 - |x|^2$ as a coordinate. Therefore, we
arrive at the coordinates 
\[
\begin{split}
&(\rho, \theta', |x|,  \theta)\in \mathbb{R} \times\mathbb{S}^{1}\times\{|x| \in  \mathbb{R}^+;  \; |x| \ge \sqrt{-\rho}~, \; \mbox{for } \rho \le 0 \}  \times\mathbb{S}^{1} \\[.3em]
&(x,y) = (|x|\theta, \sqrt{\rho + |x|^2}\theta')~.\\[.3em]
\end{split}
\]
Writing 
\[
\widetilde{W}(|x|, |y|) = \widetilde{W}( |x|, \sqrt{\rho + |x|^2}) = \int\limits_{\mathbb{S}^{1}\times \mathbb{S}^{1}}
W\circ \Phi(|x|\theta,\sqrt{\rho + |x|^2}\theta')d\theta d\theta^\prime~,
\]
then   $\mathcal{I}(\rho)$ is given by
\[
\mathcal{I}(\rho) = \chi(\rho) \times
\begin{cases}\displaystyle
\int\limits_0^\infty  \widetilde{W}(|x|, \sqrt{\rho + |x|^2}) |x|\, d|x| \quad \mbox{for } \rho \ge 0~,\\[2em]
\displaystyle
\int\limits_{\sqrt{-\rho}}^\infty \widetilde{W}(|x|, \sqrt{\rho + |x|^2})|x|\, d|x|  \quad \mbox{for } \rho \le 0~.
\end{cases}
\]

Since 
\begin{align*}
& \int\limits_0^\infty \left|\partial_{|y|} \widetilde{W}(|x|, \sqrt{\rho + |x|^2}) \frac{|x|}{\sqrt{\rho + |x|^2}}\right| \, d|x|  \lesssim\<K\>^{-\ell}, \quad \mbox{for } 0\le \rho \le  \frac12~,\\
& \int\limits_{\sqrt{-\rho}}^\infty \left|\partial_{|y|} \widetilde{W}(|x|, \sqrt{\rho + |x|^2})\frac{|x|}{\sqrt{\rho + |x|^2}}\right| \, d|x| \lesssim\<K\>^{-\ell}, \quad \mbox{for } - \frac12\le \rho \le  0,\\
&\left| \widetilde{W}(|x|,0)  \right| \lesssim\<K\>^{-\ell}~,
\end{align*}
then we conclude  that $\mathcal{I} \in Lip$ and
\[
\sup_{|\rho|\le 1/2} |\partial_\rho \mathcal{I}(\rho)| \lesssim \langle K\rangle^{-\ell}\prod_{i=1}^3 \|f_i\|_{X^{\ell,1}}~.
\]
Thus Corollary \ref{cor:dirac} gives the stated bound.
\endproof

\subsection{Bound on $I(r,c)$  for $c\neq 0$} 
Here we extend the analysis of Section \ref{sec:res-bound} to obtain  bounds  on $I(r,c)$ that decay in $r$ for  $c\ne 0$.
This is the content of the following lemma.

Recall from \eqref{penguin} that $I(r,c)$ can be written as 
\[
I(r,c) =  \int \widehat{h}_\chi(r,s) v\left(s, -\frac cr\right) \, ds
\]
where $u$ is a solution of the PDE
\begin{equation}\label{eq:hypsch3}
\begin{split}
&\partial_s \widehat{v}(s,\xi) = 2\pi i \omega(\xi)  \widehat{v}(s,\xi) \\
& \widehat{v}(0,\xi) = W(\xi).
\end{split}
\end{equation}

\medskip
\begin{lemma} \label{lem:icr} 
 Assume  $W$ is as in \eqref{eq:wz} with $f_j\in X^{\ell+n+2, N}(\RR^n)$ for $\ell>2n$ and $N>3n+2$.
  \begin{enumerate}
\item 
For every $0\le \alpha <1$ and  $c\neq 0$,
\begin{equation}\label{eq:sumcneq0}
\sum_{c\ne 0} |c|^\alpha  |I(r, c)| \lesssim r^{n-1} \langle K\rangle^{-\ell} \prod_{i=1}^3 \|f_i\|_{X^{\ell+n+2, 3n+2}(\RR^n)}
\end{equation}

\item The estimate above holds with $I(r, c)$ replaced by  $r\partial_r I(r,c)$ for $c\neq 0$.
\end{enumerate}
\end{lemma}
\begin{proof} \emph{(1)} Assume  $\|f_i\|_{X^{\ell+n+2, N}(\RR^n)}=1$.
The bound will be obtained by using the estimate in Proposition \ref{dispersive lemma}. If $c$ is distinct from $(0,0)$, $c\neq (0, [2Krs])$, $([2Krs], 0)$, or $([2Krs],[2Krs])$, there holds
\[
\abs{v\left(s, -\frac{c}{r}\right)} \lesssim r^{n+1} \langle s \rangle \langle rs \rangle^{n+1} \langle K \rangle^{-\ell-n-2}\sum_{i=0}^3 A_i~.
\]
If $c=(0, [2Krs])$ or $([2Krs], 0)$ or $([2Krs],[2Krs])$, but is distinct from zero, then we necessarily have that $|K|>(rs)^{-1}$, in which case we use estimate \eqref{PrecisedvEst-0} to get
\begin{equation*}
\abs{v\left(s, -\frac{c}{r}\right)}\lesssim \langle s \rangle^{-n}\langle K \rangle^{-\ell-n-2}\lesssim \langle s \rangle^{-n} |rs|^{n+1} \langle K \rangle^{-\ell-1}~.
\end{equation*}
Thus, by losing weights, we reach the following estimate for all $c\ne 0$,
\begin{equation}\label{uniform est on u}
\abs{v\left(s, -\frac{c}{r}\right)}\lesssim r^{n+1} \langle s \rangle \langle rs \rangle^{n+1} \langle K \rangle^{-\ell-1}\sum_{i=1}^3A_i~..
\end{equation}
From Corollary \ref{cor:Fh} we have $|\widehat h_\chi(r,s)| \lesssim \langle rs\rangle^{-M}$ for any $M$, thus 
\[
\sum_{c\ne 0} |c|^\alpha |I(r,c)| \lesssim  \langle K \rangle^{-\ell-1}   \int\frac{r^{n+1} \langle s \rangle \langle rs \rangle^{n+1}}{\langle  rs\rangle^M}   \sum_{c\ne 0}  \sum_{i=1}^3 |c|^\alpha A_i \, ds~.
\]
However, it is easy to check that $ \sum_{c\ne 0}  \sum_{i=1}^3 |c|^\alpha A_i  \lesssim  \langle Krs\rangle^\alpha$, consequently
\[
\sum_{c\ne 0} |c|^\alpha |I(r,c)|  \lesssim r^{n-1}\langle K \rangle^{-\ell}~.
\]

\emph{(2)}  Notice that $r\partial_r h(r,y)$ has the same form as that in part {\it (ii)} of Lemma \ref{hummingbird}, as a result the same estimates hold for it's Fourier transform $|\mathcal F_y r\partial_r h (r, s)|\lesssim \langle rs\rangle^{-M}$ for any $M$. Similarly, applying $r\partial_r$ to $u\big(s,\frac{-c}{r}\big)$ has the effect of replacing $u$ by $x.\nabla u$ which satisfies the same type of estimate as \eqref{uniform est on u} assuming even more decay on $f_j$ as stated in the theorem.
\end{proof}

\subsection{Proof of Theorem \ref{th:asymptotics}}

Without loss of generality, we may assume that $\|f_i\|_{X^{\ell+n+2, N}(\RR^n)}=1$ for $j=1,2,3$. Recall that
\begin{equation}\label{eq:splitting of R}
\begin{split}
\sum_{\mathcal R(K)} f_1(K_1)f_2(K_2) f_3(K_3)=& \frac{L^{2n-2}}{C_L} \sum_{q=1}^L \sum_{c\in \Z^{2n}} S(q, c) \frac{1}{q^{2n}} I(r,c)\\
=&\frac{L^{2n-2}}{C_L} \left(\sum_{q=1}^L  S(q, 0) \frac{1}{q^{2n}} I(r,0)+\sum_{q=1}^L \sum_{c\neq 0 \in \Z^{2n}} S(q, c) \frac{1}{q^{2n}} I(r,c)\right)\\
=&\frac{L^{2n-2}}{C_L}(\mathcal{A}+\mathcal{B})~.
\end{split}
\end{equation}

\underline{First case:  $n>2$.}\;   Using successively Lemma~\ref{lem:boundI}, Lemma~\ref{weightedSqcSum}, and Lemma~\ref{c:trivial_sq_est},
\begin{align*}
\mathcal{A}:= &\sum_{q=1}^L S(q,0)\frac{1}{q^{2n}} I(r,0) = \mathcal{I}(0)\sum_{q=1}^L S(q,0)\frac{1}{q^{2n}}  +
O\left( \langle K \rangle^{-\ell}\frac 1L\sum_{q=1}^L S(q,0)\frac{1}{q^{2n-1}}\right)  \\
= & \frac{\zeta\left(n-1\right)}{\zeta\left(n\right)}\mathcal{T}(f_1, f_2, f_3)+
\begin{cases}
O\left(L^{-1}(\log L)  \langle K \rangle^{-\ell} \right) \quad \text{ if } n=3\\
O\left(L^{-1}  \langle K \rangle^{-\ell} \right) \quad \text{ if } n>3
\end{cases}
\end{align*}

Turning to $\mathcal{B}$, we estimate using Lemma \ref{lem:icr}  and Lemma~\ref{c:trivial_sq_est},
\begin{align*}
|\mathcal{B}|\leq \sum_{c \neq 0} \sum_{q=1}^L |S(q,c)| \frac{1}{q^{2n}}\, |I(r,c)| & \lesssim \sum_{q=1}^L q^{-n+1} \sum_{c \neq 0} |I(r,c)|\\
&\lesssim    \sum_{q=1}^L q^{-n+1}\left(\frac{q}{L}\right)^{n-1}\langle K\rangle^{-\ell}\lesssim L^{-n+2} \langle K\rangle^{-\ell}~.
\end{align*}
This finishes the proof in the case $n>2$.

\underline{Second case: $n=2$.}  Again without any loss of generality, we assume that $\|f_j\|_{X^{\ell+6, 9}(\RR^2)}=1$ for $j=1,2,3$. We start  splitting the sum, as in \eqref{eq:splitting of R},  into $\mathcal A$ (corresponding to $c=0$) and $\mathcal B$ (corresponding to $c\ne 0$).  We find the asymptotic of $\mathcal B$ 
\[
\mathcal{B}:=  \sum_{c\neq 0 \in \Z^{4}} \sum_{q=1}^L  S(q, c) \frac{1}{q^{4}} I(r,c)~,
\]
by using Abel's summation  formula \eqref{e:Abel_summation}.  For any $x\in\RR$ we define 
\[
A(x,c) = \sum_{q'\le x} S(q',c),
\]
then 
\[
\mathcal{B} = - \sum_{c\neq 0} \int_1^L  A(x,c) \partial_x\left(\frac 1{x^4} I\left(\frac xL,c\right) \right) dx~,
\]
since $A(1,c)=I(1,c) =0$. By changing variables $x\to Lx$, we get,
\[
\mathcal{B} = \sum_{c\neq 0} \left[ 4\int_{\frac 1L}^1  A(Lx,c) \frac 1{L^4x^5} I(x,c) dx  - \int_{\frac 1L}^1  A(Lx,c) \frac 1{L^4x^4} \partial_xI(x,c) \; dx \right]~.
\]
By Lemma \ref{l:SLC_Omega_0} and Lemma \ref{l:SLC_Omega_not_0} we have
\[
A(q,c) = \sum_{q'\le q} S(q',c) = \frac{\eta(c)q^4}{4\zeta(2)}+O(q^{\frac72+\epsilon}+|c|^{\epsilon}q^{3+\epsilon})~,
\]
where
\[
\eta(c)= \begin{cases}
 1 & \text{ if } \omega(c)=0~, \\
0 & \text{ if } \omega(c)\neq 0~.
\end{cases}
\]
By Lemma \ref{lem:icr} equation \eqref{eq:sumcneq0}, we conclude
\[
4 \sum_{c\neq 0} \int_{\frac 1L}^1  A(Lx,c) \frac 1{L^4x^5} I(x,c) dx = \sum_{c\neq 0}   \frac{\eta(c)}{\zeta(2)}\int_0^1  \frac 1{x} I(x,c) dx  + O(L^{-1/2+} \< K\>^\ell)~.
\]

Similarly,
\[
\int_{\frac 1L}^1  A(Lx,c) \frac 1{L^4x^4} \partial_xI(x,c) \; dx =  O(L^{-1/2+} \< K\>^\ell),
\]
where we used Lemma \ref{lem:icr} \emph{(ii)}. This implies that for $n=2$, we have
\begin{equation}\label{eq:2ndterm}
\mathcal{B}=\sum_{c \neq 0} \sum_{q=1}^\infty S(q,c) \frac{1}{q^4}\, I(r,c) = \sum_{c \neq 0,\omega(c)=0} \frac{1}{\zeta(2)} \int_0^1 \frac 1r I(r,c)\, dr  + O\left (L^{-1/2+}\langle K \rangle^{-\ell}\right)~.
\end{equation}
To bound $\mathcal{A}$, we proceed as in  Heath-Brown \cite{HB}, and split the sum into 
\[
\sum_{q=1}^L S(q,0)\frac{1}{q^4} I(r,0) = \sum_{q=1}^{\rho L} S(q,0)\frac{1}{q^4} I(r,0)  +  \sum_{q=\rho L}^L S(q,0)\frac{1}{q^4} I(r,0)= I + II~,
\]
where  $\rho = L^{-\alpha}$ is to be chosen shortly. To estimate $I$ we write
\[
\sum_{q=1}^{\rho L} S(q,0)\frac{1}{q^4} I(r,0) =  \sum_{q=1}^{\rho L} S(q,0)\frac{1}{q^4} \mathcal{I}(0) +  \sum_{q=1}^{\rho L} S(q,0)\frac{1}{q^4} \left( I(r,0)- \mathcal{I} (0)\right)~,
  \]
 and use equation  \eqref{eq:lnb}, equation \eqref{eq:sqcb}, Lemma \ref{lem:boundI} and \eqref{aim0} to obtain
\begin{equation}\label{eq:firstlog}
 \sum_{q=1}^{\rho L} S(q,0)\frac{1}{q^{4}} I(r,0)  = \left (\frac{\log(\rho L)+\gamma}{\zeta(2)}-\frac{\zeta'(2)}{\zeta(2)^2}
\right) \mathcal{I}(0)+O\left((\rho L)^{-1/2+}\langle K \rangle^{-\ell}+\rho \langle K \rangle^{-\ell}\right)~.
 \end{equation}

To estimate $II$, we use Abel's summation formula again to write
\begin{align*} 
& \sum_{q=\rho L}^L S(q,0)\frac{1}{q^4} I(r,0) \\
& \qquad = -\frac{A(\rho L,0)I(\rho,0)}{(\rho L)^4} +  \int_{\rho }^1 \left(A(Lx,0)\frac{4}{L^4x^{5}} I(x,0) -  
A(Lx,0)\frac{1}{L^4x^{4}} \partial_x I(x,0)\right) \, dx.
\end{align*}

Using the formula for $A(q,0)$ given in Lemma  \ref{l:SLC_Omega_0},  we conclude
\[
\begin{split}
\frac{A(\rho L,0)I(\rho,0)}{(\rho L)^4}=&\frac{1}{4\zeta(2)}\mathcal I(0)+\frac{1}{\zeta(2)}\left( I(\rho, 0)-\mathcal I(0)\right) + O\left((\rho L)^{-\frac12+}\langle K \rangle^{-\ell}\right)\\
=&\frac{1}{4\zeta(2)}\mathcal I(0) + O\left(\rho \langle K \rangle^{-\ell}+(\rho L)^{-\frac12+}\langle K \rangle^{-\ell}\right)~.
\end{split}
\]
Similarly, we have
\[
\begin{split}
\int_{\rho}^1 A(Lx,0)& \frac{4}{L^4 x^{5}} I(x,0) dx  = \frac{1}{\zeta (2)}\int_{\rho }^1 \frac{1}{x} I(x,0)\,dx+\ O \left((\rho L)^{-\frac 12 +} \right)\\
=&-\frac{1}{\zeta (2)}\log \rho \,\mathcal I (0) +\frac{1}{\zeta (2)}\int_{0}^1 \frac{1}{x} \left(I(x,0)-\mathcal I(0)\right)\,dx + O \left((\rho L)^{-\frac 12 +} \<K\>^\ell\right)
\end{split}
\]
and
\[
\int_{\rho }^1 A(Lx,0)\frac{1}{L^4x^{4}} \partial_x I(x,0) \, dx
 =-\frac{1}{4\zeta (2)} \mathcal I(0)+ O\left((\rho L)^{-\frac 12 +} \langle K \rangle^{-\ell}+\rho \langle K \rangle^{-\ell}\right)~.
 \]

Putting all of these terms together we conclude
\begin{equation}\label{eq:secondlog}
\begin{split}
\sum_{q=\rho L}^L S(q,0)\frac{1}{q^4} I(r,0)  = &-\frac{1}{\zeta (2)}\log \rho \,\mathcal I (0) + \frac{1}{\zeta (2)}\int_{0}^1 \frac{1}{x} \left(I(x,0)-\mathcal I(0)\right)\,dx\\
&  \qquad + O\left((\rho L)^{-\frac 12 +} \langle K \rangle^{-\ell}+\rho \langle K \rangle^{-\ell}\right)~.
\end{split}
\end{equation}

Adding equations \eqref{eq:firstlog} and \eqref{eq:secondlog}, and choosing $\rho =L^{-\frac13}$, we get
\begin{multline*}
%\begin{split}
\mathcal{A}:=\sum_{q=1}^L S(q,0)\frac{1}{q^{4}} I(r,0) =  \left(\frac{\log(L)+\gamma}{\zeta(2)}-\frac{\zeta'(2)}{\zeta(2)^2}
\right) \mathcal{I}(0)  \\
+\frac{1}{\zeta (2)}\int_{0}^\infty \frac{1}{r} \left(I(r,0)-\mathcal I(0)\right)\,dr + O\left(L^{-\frac 13 +}\right)~.
%\end{split}
\end{multline*}
Recalling that $\mathcal I(0)=\mathcal T(f_1, f_2, f_3)$, and combining the above equation to \eqref{eq:2ndterm}, we obtain our claim with
\begin{multline}\label{def of C(K)}
\mathcal C(K)=\left(\frac{\gamma}{\zeta(2)}
-\frac{\zeta'(2)}{\zeta(2)^2} \right) \mathcal{I}(0) \\
+  \frac{1}{\zeta(2)}\int_{0 }^1 \frac{1}{r} \left( I(r,0) - \mathcal{I}(0) \right) \,dr +\sum_{c \neq 0,\omega(c)=0} \frac{1}{\zeta(2)} \int_0^1 \frac 1r I(r,c)\, dr~.
\end{multline}

Finally, the boundedness of $\mathcal C$ as stated follows from \eqref{aim0}, Lemma \ref{lem:boundI}, and Lemma \ref{lem:icr}.

\section{Normal Form Transformations and proof of Theorems~\ref{approx thm1} and  \ref{approx thm2} }\label{sec:NF-section}

The proof of both theorems has two main ingredients: {\it 1)}~asymptotics of lattice sums stated in Theorem \ref{th:asymptotics}, and {\it 2)}~normal forms transformation. 
We first explain the latter before proving the theorems.

\subsection{Normal Form Transformation} The normal form transformation can be derived either by using the method of averaging, or by calculating a coordinate change derived from a power series expansion.  Here we elect to use the latter. Let 
\begin{equation}\label{eq:NF}
v= u + \sum_{d=1}^P  \epsilon^{2d} H_{2d+1}(u)~,
\end{equation}
where $H_{2d+1}$ is a 2d+1 multilinear form in $u$ (in odd entries) and $\bar{u}$ (in even entries).

Recall that $u$ solves~\eqref{NLS}; therefore $v$ satisfies
\begin{multline*}
-i\partial_t v + \frac{1}{2\pi}\Delta v = \epsilon^2 |u|^2u +
 \sum_{d=1}^P  \frac{\epsilon^{2d}}{2\pi}\Delta H_{2d+1}(u)\\
+ \sum_{d=1}^P  \epsilon^{2d} \frac{\delta H_{2d+1}}{\delta
  u}\bigg(\frac{-1}{2\pi}\Delta u\bigg)  + \epsilon^{2d+2} \sum_{d=1}^P \frac{\delta H_{2d+1}}{\delta
  u}(|u|^2 u)~.
\end{multline*}
where we used the notation for any function $F$ depending on $u$ and $\bar{u}$
\[
\frac{\delta F}{\delta u}(w) = \frac{\partial F}{\partial u} w +  \frac{\partial F}{\partial \bar u} \bar w~.
\]

%%jalal u and u bar
Writing $\mathcal{L}H_{2d+1}(u) = \frac{1}{2\pi}\Delta H_{2d+1}(u) - \frac{\delta H_{2d+1}}{\delta
  u} (\frac{1}{2\pi}\Delta u)$, and collecting terms of the same order in $\epsilon$, we conclude that
\begin{align*}
-iv_t + \frac{1}{2\pi}\Delta v = \epsilon^2|u|^2u  +
\epsilon^2\mathcal{L}H_3(u) &
+ \sum_{d=2}^P \epsilon^{2d} \left( \mathcal{L}H_{2d+1}(u) + \frac{\delta H_{2d-1}}{\delta
  u}(|u|^2 u)\right)  \\[.3em]
&  +\ \epsilon^{2(P+1)} \frac{\delta H_{2P+1}}{\delta u} (|u|^2 u)~.
\end{align*}
To express the above equation in terms of the Fourier coefficients in the following manner, recall that
\[
u = \frac{1}{L^n} \sum_{K \in \mathbb{Z}^n_L} u_K e(K\cdot x), \qquad v = \frac{1}{L^n} \sum_{K \in \mathbb{Z}^n_L} v_k e(K\cdot x)~,
\]
and write the multilinear form $H_{2d+1}$  as
\begin{align*}
& H_{2d+1}(u) = \frac{1}{L^n} \sum_{K \in \mathbb{Z}^n_L} H_{2d+1,K}(u) e(K\cdot x) \\
\mbox{where}\quad H_{2d+1,K}(u)= \frac{1}{L^{2nd}}
& \sum_{\mathcal{S}_{2d+1}(K) =0} h_{2d+1}(K, K_1,
  K_2,\dots,K_{2d+1})u_{K_1} \bar{u}_{K_2} \dots u_{K_{2d+1}}~.
\end{align*}
The equation for $v_K$ can then be written as
\begin{multline}\label{eq:v}
-i\partial_t v_K - 2\pi|K|^2v_K =  \frac {\epsilon^2}{L^{2n}}
\sum_{\mathcal{S}_3(K)=0} u_{K_1} \bar{u}_{K_2} u_{K_3}\\
+ \sum_{d=2}^P \epsilon^{2d} \left(\mathcal{L} H_{2d+1,K}(u) +\frac{\delta H_{2d-1,K}}{\delta u} (|u|^2 u)\right)  + \epsilon^{2(P+1)}\frac{\delta H_{2P+1,K}}{\delta u} (|u|^2 u)~,
\end{multline}
where $\mathcal{L} H_{2d+1,K}(u)$ and $\frac{\delta H_{2d-1,K}}{\delta u} (|u|^2 u)$ are given by
\begin{equation*}
\mathcal{L} H_{2d+1,K}(u)= \frac{1}{L^{2nd}}
\sum_{\mathcal{S}_{2d+1}(K)=0}  \hskip -12pt
 2\pi \Omega_{2d+1}(K,K_1,\dots,K_{2d+1})h_{2d+1}(K,K_1,\dots,K_{2d+1})
u_{K_1}\dots u_{K_{2d+1}}~,
\end{equation*}
and
\begin{equation}\label{eq:hk}
\begin{split}
\frac{\delta H_{2d-1,K}}{\delta u} (|u|^2 u)(K) = &
\frac{1}{L^{2nd}} \sum_{\mathcal{S}_{2d-1}(K)=0} h_{2d-1}(K,K_1,\dots,K_{2d-1})\times\\
& \qquad \left(\sum_{N-K_{2d} + K_{2d-1} = K_1} u_N
\bar{u}_{K_{2d}}u_{K_{2d+1}}\right) \times \bar{u}_{K_2}u_{K_3}\dots u_{K_{2d-1}}\\[.5em]
& + \frac{1}{L^{2nd}} \sum_{\mathcal{S}_{2d-1}(K)=0}  h_{2d-1}(K,K_1,\dots,K_{2d-1}) \times\\
& \qquad \left(\sum_{N-K_{2d} + K_{2d-1} = K_2} \bar{u}_N
u_{K_{2d}} \bar{u}_{K_{2d+1}}\right) \times u_{K_1}u_{K_3}\dots u_{K_{2d-1}}\\[.5em]
&+ \dots 
\end{split}
\end{equation}
The normal form transformation is determined by choosing $H_{2d+1}$ to eliminate non-resonant terms of degree $2d+1$,  i.e., those for which the resonance modulus $\Omega$ does not vanish. This leads to the choice
\begin{align*}
H_{3,K}(u)  = & \frac{-1}{L^{2n}} \sum_{\substack{\mathcal{S}_3
  (K)=0 \\ \Omega_3(K) \neq 0}}\frac{1}{2\pi \Omega_3(K,K_1,K_2,K_3)} u_{K_1}\bar{u}_{K_2} u_{K_3}  \\
 H_{2d+1,K}(u)  = &\frac{-1}{L^{2nd}} \sum_{\mathcal{S}_{2d-1}(K)=0}h_{2d-1}(K,K_1,\dots,K_{2d-1})  \times  \\
& \hskip -3pt
\left(\sum_{\substack{N-K_{2d}+K_{2d+1} = K_1 \\ \Omega_{2d+1} \neq 0}}\frac{1}{2\pi \Omega_{2d+1}(K,N,K_2,\dots,K_{2d+1})} u_N \bar{u}_{K_{2d}}u_{K_{2d+1}}\right)  \bar{u}_{K_2}u_{K_3}\dots u_{2d-1}\\
& +\dots,
\end{align*}
which defines iteratively $H_{2d+1}$. With this choice for $H_{2d+1}$, equation \eqref{eq:v} becomes
\begin{multline}\label{eq:vnf}
-i\partial_t v_K - 2\pi|K|^2v_K =  \frac{\epsilon^2}{L^{2n}}
\sum_{\substack{\mathcal{S}_3(K)=0\\ \Omega_3 = 0}} u_{K_1} \bar{u}_{K_2} u_{K_3} + \sum_{d=2}^P \epsilon^{2d}
\sum_{\substack{\mathcal{S}_{2d+1}(K)=0 \\ \Omega_{2d+1}=0}}
\frac{\delta H_{2d-1,K}}{\delta u} (|u|^2 u) \\
+ \epsilon^{2(P+1)} \frac{\delta H_{2P+1,K}}{\delta u} (|u|^2 u)~,
\end{multline}
where we abused notations slightly be denoting $\sum_{\substack{\mathcal{S}_{2d+1}(K)=0 \\ \Omega_{2d+1}=0}}$ the sum restricted to the Fourier modes satisfying the resonance condition.

Finally, the boundedness properties of the normal form transformation~\eqref{eq:NF} will be needed.

\begin{lemma} \label{boundNF} For any $d\leq P$,
$$
\| H_{2d+1} (u) \|_{X^\ell} \lesssim L^+ \| u \|_{X^\ell}^{2d+1}~.
$$
As a result, for $v$ given by~\eqref{eq:NF},
$$
\| v - u \|_{X^\ell} \lesssim L^+ \sum_{d=1}^P \epsilon^{2d} \| u \|_{X^\ell}^{2d+1}~.
$$
\end{lemma}

\begin{proof} It suffices to show that
$$
\frac{1}{L^{2nd}} \sum_{\mathcal{S}_{2d+1}(K)=0} |h_{2d+1}(K,K_1,\dots,K_{2d+1})| \langle K_1 \rangle^{-\ell} \dots \langle K_{2d+1} \rangle^{-\ell} \lesssim L^+ \langle K \rangle^{-\ell}~.
$$
Start with $d=1$, for which $h_3 = \frac{1}{\Omega_3}$. We bound the sum by writing  $\Omega_{3} = \frac{\mu}{L^2}$ and  splitting the sum over $|\mu| \leq L^{10n}$ and  $|\mu| > L^{10n}$. 
For $|\mu| > L^{10n}$, the sum can be bounded directly by
$$
\frac{1}{L^{2n}} \sum_{\substack{K_1-K_2+K_3=K \\ |\Omega_3| > L^{10n-2}}} \frac{1}{L^{10n-2}} \langle K_1 \rangle^{-\ell}\langle K_2 \rangle^{-\ell}\langle K_3 \rangle^{-\ell} \lesssim \frac{\langle K \rangle^{-\ell}}{L^{8n-2}}~.
$$
For $|\mu| \leq L^{10n}$, we use Theorem \ref{th:disbnd} to bound
\[
\frac{1}{L^{2n}}\sum_{1 \leq |\mu| \leq L^{10n}} \frac{L^2}{\mu} \sum_{\substack{K_1-K_2+K_3=K \\ \Omega_3 = \mu}}\langle K_1 \rangle ^{-\ell}\langle K_{2}\rangle^{-\ell}\langle K_{3}\rangle^{-\ell} \lesssim \frac{1}{L^{2n}} L^{2+} Z_n(L) \langle K \rangle^{-\ell} \lesssim L^{+} \langle K \rangle^{-\ell}~.
\]
Turning to $d \geq 2$, we use the recursive definition to write
\begin{align*}
& \frac{1}{L^{2nd}} \sum_{\mathcal{S}_{2d+1}(K)=0} |h_{2d+1}(K,K_1,\dots,K_{2d+1})| \langle K_1 \rangle^{-\ell} \dots \langle K_{2d+1} \rangle^{-\ell}  \lesssim \\
& \frac{1}{L^{2nd}}\hskip -2pt
 \sum_{\mathcal{S}_{2d-1}(K)=0} |h_{2d-1}|  
\left(\sum_{\substack{N-K_{2d}+K_{2d+1} = K_1 \\ \Omega_{2d+1} \neq 0}}
\frac{1}{\left |\Omega_{2d+1}\right |} \langle N \rangle^{-\ell}\langle K_{2d} \rangle^{-\ell}\langle K_{2d+1} \rangle^{-\ell}\right) \langle K_2 \rangle^{-\ell} \dots \langle K_{2d-1} \rangle^{-\ell}~.
\end{align*}
The inner sum can be bounded as above, leading to the desired estimate:
$$
\dots \lesssim \frac{L^+}{L^{n(2d-2)}} \sum_{\mathcal{S}_{2d-1}(K)=0} |h_{2d-1}| \langle K_2 \rangle^{-\ell} \dots \langle K_{2d+1} \rangle^{-\ell}\langle K_1 \rangle^{-\ell} \dots \langle K_{2d-1} \rangle^{-\ell} \lesssim L^+ \langle K \rangle^{-\ell}~,
$$
where the last inequality follows by the bound at the rank $d-1$.
\end{proof}

\subsection*{Proofs of Theorems  \ref{approx thm1}  and  \ref{approx thm2}}

The idea of the proof is to compare solutions of the \eqref{NLS} and the \eqref{CR} equations using
the transformed equation  \eqref{eq:v}  and Lemma~\ref{boundNF}.  
\endproof
\underline{Proof of  Theorem  \ref{approx thm1} }
Let
\[
a_K(t) =u_K(t) e(-|K|^2 t)~\mbox{and}~  \qquad b_K(t) = v_K(t) e(-|K|^2 t)~.
\]
Then, from Lemma~\ref{boundNF} we have
\[
\begin{split}
\left\| a_K(t) - g\left(\frac{t}{T_R},K\right)\right\|_{X^\ell} \le &  \left\| b_K(t) - g\left(\frac{t}{T_R},K\right)\right\|_{X^\ell}  + \| a_K(t) -  b_K(t) \|_{X^\ell} \\
 \le &  \left\| b_K(t) - g\left(\frac{t}{T_R},K\right)\right\|_{X^\ell} + C_\gamma\left(\sum_{d=1}^P
\epsilon^{2d}   \|u\|^{2d+1}_{X^{\ell}} \right)  L^{\gamma}~,
\end{split}
\]
where from the hypothesis of the theorem
\begin{equation*}%\label{eq: gepsilon}
\begin{split}
&-i\partial_t g = \mathcal T(g, g, g)\\
&\sup_{0\leq t \leq M} \|g(t)\|_{X^{\ell+n+2, N}(\mathbb{R}^n)}\leq B~,
\end{split}
\end{equation*}
and   $T_R =\frac{L^{2n}}{\epsilon^2 Z_n(L)}$.  Initially $a_K(0) = g(0,K)$, and we will show by a bootstrap argument that $\|a_K(t)\|_{X^\ell} \le 2B$ on the time interval $[0,MT_R]$ stated in the theorem. Consequently, we can assume that
\begin{equation}\label{eq:bst}
\left\| a_K(t) - g\left(\frac{t}{T_R},K\right)\right\|_{X^\ell} \le   \left\| b_K(t) - g\left(\frac{t}{T_R},K\right)\right\|_{X^\ell} + C_\gamma\left(\sum_{d=1}^P
\epsilon^{2d}  B^{2d+1}\right)  L^{\gamma}.
\end{equation}
To bound $w_K := b_K(t) - g(\frac{t}{T_R},K)$, we write equation \eqref{eq:vnf} in terms of $b_K(t) $
\begin{align*}
-i\partial_t b_K  = & \frac{\epsilon^2}{L^{2n}}
\sum_{\substack{\mathcal{S}_3(K)=0\\ \Omega_3 = 0}} a_{K_1} \bar{a}_{K_2} a_{K_3} + \sum_{d=2}^P \epsilon^{2d}
\sum_{\mathcal{S}_{2d+1}(K)=0}
\frac{\delta \tilde{H}_{2d-1,K}}{\delta u} (|a|^2 a) \\
&+ \epsilon^{2(P+1)} e(-K^2t)\frac{\delta H_{2P+1,K}}{\delta u} (|u|^2 u) ,
\end{align*}
where we used the fact that $\Omega_3= \Omega_{2d+1}=0$ and we wrote
\[
\begin{split}
\frac{\delta \tilde{H}_{2d-1,K}}{\delta u} (|a|^2 a)(K) =&
\frac{1}{L^{2nd}} \sum_{\mathcal{S}_{2d-1}(K)=0} h_{2d-1}\times\\
&\left(\sum_{\substack{N-K_{2d} + K_{2d+1} = K_1\\ \Omega_{2d+1}=0} }a_N\bar{a}_{K_{2d}} a_{K_{2d+1}}\right) \times \bar{a}_{K_2}\dots b_{K_{2d-1}}\\[.5em]
&+ \dots.
\end{split}
\]
This implies that
\begin{equation}\label{eq:w123}
\begin{split}
-i\partial_t w_K  = & \frac{\epsilon^2}{L^{2n}}
\left( \sum_{\substack{\mathcal{S}_3(K)=0\\ \Omega_3 = 0}} a_{K_1} \bar{a}_{K_2} a_{K_3} - Z_n(L) \mathcal{T}(g)_K\right) \\
& + \sum_{d=2}^P \epsilon^{2d}
\sum_{\mathcal{S}_{2d+1}(K)=0}
\frac{\delta \tilde{H}_{2d-1,K}}{\delta u} (|a|^2 a) + \epsilon^{2(P+1)} e(-K^2t)\frac{\delta H_{2P+1,K}}{\delta u} (|u|^2 u) \\
&:= I + II + III~.
\end{split}
\end{equation}

\underline{Bound on $I$}  
\[
\begin{split}
I= & \frac{\epsilon^2}{L^{2n}}
\ \sum_{\substack{\mathcal{S}_3(K)=0\\ \Omega_3 = 0}} \left(a_{K_1} \bar{a}_{K_2} a_{K_3} - g(K_1)\bar{g}(K_2)g(K_3) \right)\\
& +   \frac{\epsilon^2}{L^{2n}}\left( \sum_{\substack{\mathcal{S}_3(K)=0\\ \Omega_3 = 0}}g(K_1)\bar{g}(K_2)g(K_3)
- Z_n(L) \mathcal{T}(g)\right) 
=I_1 + I_2~.
\end{split}
\]

By Theorem \ref{th:disbnd},
\[
\| I_1\|_{X^\ell} \le C_0  \frac{\epsilon^2 Z_n(L)}{L^{2n}}B^2 \left\| a_K(t) - g\left(\frac{t}{T_R},K\right)\right\|_{X^\ell} ,
\]
and by Theorem \ref{th:asymptotics},
\[
I_2 \le CB^3\frac{\epsilon^2 Z_n(L)}{L^{2n}} \delta(L) ;  \qquad \text{where} \qquad \delta(L):=
\begin{cases} L^{-1+} \qquad &\text{ if }n\geq 3\\
(\log L )^{-1} \qquad &\text{ if } n=2
\end{cases}
\]

\underline{Bound on $II$}  

By Theorem  \ref{th:disbnd}  we have
\[
\left |\frac{1}{L^{2n}}\sum_{\substack{N-K_{2d} + K_{2d+1} = K_1\\ \Omega_{2d+1}=0} }a_N\bar{a}_{K_{2d}} a_{K_{2d+1}}\right | \lesssim \frac{Z_n(L)}{L^{2n}} B^3\langle K_1\rangle^{-\ell}~,
\]
for any $d$, and consequently from Lemma~\ref{boundNF} we deduce
\[
\left\|\frac{\delta \bar{H}_{2d-1,K}}{\delta u} (|a|^2 a)(K)\right\|_{X^{\ell}} \lesssim B^{2d+1}  \frac{Z_n(L)}{L^{2n} }L^\gamma~,
\]
and 
\[
\|II\|_{X^\ell} \le C_\gamma \left(\sum_{d=2}^P \epsilon^{2d-4}B^{2d+1} \right)  \frac{\epsilon^2 Z_n(L)}{L^{2n} } \epsilon^2 L^\gamma~.
\]
\underline{Bound on $III$}  
This term can be bounded directly using Lemma \ref{boundNF} 
\[
\| III\|_{X^\ell} \le C_\gamma B^{2(P+1)} \epsilon^{2(P+1)} L^\gamma.
\]

Integrating \eqref{eq:w123} and using \eqref{eq:bst} we conclude 
\begin{multline*}
\left\| a_K(t) - g\left(\frac{t}{T_R},K\right)\right\|_{X^\ell}
 - C_{\gamma,B} \epsilon^2  L^{\gamma}
 \le\int_0^t  \left\{ C_0  \frac{\epsilon^2 Z_n(L)}{L^{2n}}B^2 \| a_K(s) - g(\frac{s}{T_R},K)\|_{X^\ell} \right. \\
 \left. +  C_{\gamma,B}  \frac{\epsilon^2 Z_n(L)}{L^{2n}} \delta(L) +  C_{\gamma,B}     \frac{\epsilon^2 Z_n(L)}{L^{2n} } \epsilon^2 L^\gamma+ C_{\gamma,B}  \epsilon^{2P+1} L^{\gamma} \right\} ds.  
\end{multline*}
From Gronwall's inequality, and $0\le t \le T_R M$,  we obtain,
\[
\left\| a_K(t) - g\left(\frac{t}{T_R},K\right)\right\|_{X^\ell} \le C_{\gamma,B} \left(  \epsilon^2  L^{\gamma} + \delta(L) M +  \epsilon^2 L^\gamma M
+  \epsilon^{2P+1} L^{2 + \gamma}M \right) e^{C_0B^2 M}
\]

and thus by choosing $L$ large, $\epsilon^2L^\gamma$ small, and $P$ large, we conclude
\[
\sup_{0\le t \le T_R\delta_0} \left\| a_K(t) - g\left(\frac{t}{T_R},K\right)\right\|_{X^\ell}  \le C_{\gamma,B, M} \left(  \epsilon^2  L^{\gamma} + \delta(L) \right) \le   \frac{B}{2},
\]
ensuring that $\|a_K(t)\|_{X^\ell} \le 2B$,  which completes the bootstrap argument and the proof of the theorem.
\endproof

\underline{Proof of Theorem \ref{approx thm2}.}
The proof follows from the same argument in Theorem   \ref{approx thm1}.  One only needs to replace the term $I$ in  \eqref{eq:w123}  by $\widetilde \Delta$ defined in \eqref{eq: def of tilde Delta} and using \eqref{eq: Delta bound2}. 

\section{The general case $p \in \mathbb{N}$}
\label{section_general_p}
The proof for the general problem 
\[
\tag{NLS}
- i \partial_t v + \frac{1}{2\pi} \Delta v =\epsilon^{2p}|v|^{2p} v \qquad   (t,x) \in \mathbb{R} \times \mathbb{T}^n_L ,  
\]
proceeds in exactly the same manner as the case $p=1$.  We start by writing the equation for the Fourier coefficients $a_K(t) = e(-|K|^2 t) \widehat{v}_K(t)$,

\begin{align*}
- i \partial_t a_K =& \frac{\epsilon^{2p}}{L^{2pn}}\sum_{\mathcal{S}_{2p+1}(K) =0}  
a_{K_1} \overline{a_{K_2}} a_{K_3} \dots a_{K_{2p+1}}e(2\pi\Omega_{2p+1}(K)t)\\
 =& \underbrace{\frac{\epsilon^{2p}}{L^{2pn}}\sum_{\substack{\mathcal{S}_{2p+1}(K) =0\\ \Omega_{2p+1}(K)=0}} 
a_{K_1} \overline{a_{K_2}} a_{K_3} \dots a_{K_{2p+1}}}_{\mbox{resonant interactions}} \\[.5em]
& + \underbrace{\frac{\epsilon^{2p}}{L^{2pn}}\sum_{\substack{\mathcal{S}_{2p+1}(K) =0 \\ \Omega_{2p+1}(K)\neq 0}} 
a_{K_1} \overline{a_{K_2}} a_{K_3} \dots a_{K_{2p+1}} e(\Omega_{2p+1}(K)t)}_{\mbox{non-resonant interactions}},
\end{align*}
where,
\begin{align*}
& \mathcal{S}_{2p+1}(K) = K_1 - K_2 + K_3 \dots +K_{2p+1} - K~,\\
& \Omega_{2p+1}(K) = |K_1|^2 - |K_2|^2 + |K_3|^2  \dots + |K_{2p+1}|^2 - |K|^2~. 
\end{align*}
The normal form transformation proceeds in an identical manner, leading us to consider the resonant system
\[
- i \partial_t a_K = \frac{\epsilon^{2p}}{L^{2pn}}\sum_{\substack{\mathcal{S}_{2p+1}(K) =0\\ \Omega_{2p+1}(K)=0}} 
a_{K_1} \overline{a_{K_2}} a_{K_3} \dots a_{K_{2p+1}}~.
\]
Thus  we only need to compute the asymptotics of the lattice sum
\[
\sum_{\substack{\mathcal{S}_{2p+1}(K) =0\\ \Omega_{2p+1}(K)=0}} W_{2p+1}(K)  \overset{def}{=} \sum_{\substack{\mathcal{S}_{2p+1}(K) =0\\ \Omega_{2p+1}(K)=0}} 
a_{K_1} \overline{a_{K_2}} a_{K_3} \dots a_{K_{2p+1}}~.
\]
The resonant set  ($\mathcal{S}_{2p+1}(K) =0$,   $\Omega_{2p+1}(K)=0$  ), when $K_i\in \mathbb{Z}_L^n$, is identical to the case when $p=1$ and $K_i\in \mathbb{Z}_L^{pn}$.  This can be seen by  {\it 1)}~translating $K_i\to K_i + K$; {\it 2)}~writing the resonant set as
\begin{multline*}
 |K_1|^2 - |K_2|^2 + |K_3|^2  \dots -|K_{2p}|^2+ | K_1 - K_2 + K_3 \dots -K_{2p}|^2 =\\
 2(K_1-K_2)\cdot \left(K_1 +  K_3 \dots -K_{2p}\right) + |K_3|^2  \dots -|K_{2p}|^2+ | K_3 \dots -K_{2p}|^2  =\\
 -2J_1\cdot J_2 + |K_3|^2  \dots -|K_{2p}|^2+ | K_3 \dots -K_{2p}|^2  =0~,
 \end{multline*}
where $J_1= K_1 +  K_3 \dots -K_{2p}$ and $J_2 = K_2-K_1$;  {\it 3)}~ repeating the expansion and setting 
\begin{equation}\label{eq:coordp}
\begin{split}
&J_{2i-1}=  K_{2i-1} + K_{2i+1} -K_{2i+2}\dots -K_{2p}\\[.3em]
&J_{2i}= K_{2i} - K_{2i-1}~,
\end{split}
\end{equation}
with $J_{2p-1}= K_{2p-1}$; consequently  the resonant set is given by
\[
\boldsymbol{J_e\cdot J_o} = J_1\cdot J_2 + \dots J_{2p-1}\cdot  J_{2p} = 0~,
\]
where $(\boldsymbol{J_e}, \boldsymbol{J_o}) \in \mathbb{Z}_L^{2np}$.

Now we turn to the asymptotics of the lattice sum.  We need to verify that the answers derived from the circle method are uniform in $K$, and in fact decay like $\<K\>^{-\ell}$.  This can be accomplished by noting two things.  First, as in the case when $p=1$, we have, 
\[
 \mathcal{S}_{2p+1}(K) = K_1 - K_2 + K_3 \dots +K_{2p+1} - K=0~,
 \]
 which implies $|K_{i_0}|\ge \frac{|K|}{4p}$,  for some $i_0$, and therefore,
 \[
 |W_{2p+1}(K)| \lesssim \<K\>^{-\ell} \prod_{\substack{j=1\\ j\ne i_0}}^{2p+1}\<K_j\>^{-\ell}~.
 \]
Second, to verify that the proofs of the lemmas, propositions, and theorems in Sections \ref{sec:res-bound} and \ref{sec:asymp} work for any $p\in \mathbb{N}$, we have to show that the integration by parts argument does not lead to growth in powers of $K$.  To this end we note that to prove the version of Proposition \ref{dispersive lemma} in the general case, we proceed as follows. When
$|K_{i_0}|\ge \frac{|K|}{4p}$, we substitute  
%\[
%K_{i_0} = (-1)^{i_0 -1} \left( \sum_{\substack{i=1\\ i\ne i_0}}^{2p+1} (-1)^i K_i + K\right)
%\]
\[
(-1)^{i_0} K_{i_0} = \sum_{\substack{i=1\\ i\ne i_0}}^{2p+1} (-1)^{i-1} K_i - K~.
\]
in the expression of $\Omega_{2p+1}(K)$ and write
\[
\Omega_{2p+1}(K)= \omega_{i_0} + 2 K\cdot \left( \sum_{\substack{i=1\\ i\ne i_0}}^{2p+1} (-1)^{i} K_i\right) - \left(1 + (-1)^{i_0}  \right)|K|^2~,
\]
where \underline{$\omega_{i_0}$ does not depend on either $K_{i_0}$ or $K$}.  Consequently as in the case when $p=1$, we can integrate by parts on the middle term $2 K\cdot \left( \sum_{\substack{i=1\\ i\ne i_0}}^{2p+1} (-1)^{i} K_i\right)$,  without introducing powers of $K$, provided we exclude the cases when one of $c_i = 2Krs$.  These observations are all that is needed to prove the following theorem:
\begin{theorem}
Let $f_i\in X^{\ell,N}$, and  denote by  $z= (\boldsymbol{J_e}, \boldsymbol{J_o})$, where $J_i$ are given by \eqref{eq:coordp}.  Denote by
\begin{equation*}
W(z) = f_1( K + K_1) \bar{f}_2(K +K_2)\dots f_{2p+1}( K + K_{2p+1})~,
\end{equation*}
 where the $K_i$ are considered as functions of $J_i$ by inverting \eqref{eq:coordp}, and  
\begin{equation*}
\mathcal P(W)(K)=\int_{\mathbb{R}^{2pn+1}} \delta(\Omega_{2p+1}(K))  \delta(\mathcal S_{2p+1}(K)) W(z) dK_1dK_2\dots dK_{2p+1}~,
\end{equation*} 
and set
\[
Z_{pn}(L) = \left\{
\begin{array}{ll}
\frac{1}{\zeta(2)} L^2 \log L & \mbox{if $pn=2$} \\
\frac{\zeta(pn-1)}{\zeta(pn)} L^{2pn-2} & \mbox{if $pn \ne 2$}
\end{array} \right.
\]
\begin{enumerate}
\item[1)] For $pn\ne 2$, define 
\begin{align*}
&\Delta(W) =\frac{1}{Z_n(L)}\sum_{\substack{z\in \mathbb{Z}_L^{2n}\\ \omega(z) =0}} W(z) - \mathcal T(W)~.
\end{align*}
If $f_j \in X^{\ell+n+2, 4n+2}(\RR^n)$ for $j=1,2,3$, then
\begin{equation*}
\|\Delta(W)\|_{X^\ell} \lesssim L^{-1+}\prod_{i=1}^{2p+1} \|f_i\|_{X^{\ell+n+2, 4n+2}(\RR^n)}~.
\end{equation*}

\item[2)] For $pn=2$, define
\begin{align*}
\widetilde \Delta(W) =\frac{1}{Z_2(L) }  \sum_{\boldsymbol{J_o} \cdot \boldsymbol{J_e}= 0}W(z) - \left(\mathcal P(W)  + \frac{\zeta(2)}{\log L} \mathcal C(W)\right)~,
\end{align*}
where $\mathcal C(W)$ is a correction operator that is independent of $L$ and is defined explicitly in \eqref{def of C(K)}.

If $f_j \in X^{\ell+n+4, 4n+3}(\RR^n)$, then

\begin{equation*}
\begin{split}
&\| \mathcal C(W)\|_{X^\ell} \lesssim \prod_{i=1}^{2p+1} \|f_i\|_{X^{\ell+n+4, 4n+3}(\RR^n)}~,\\
&\|\widetilde \Delta(W)\|_{X^\ell} \lesssim L^{-1/3+}\prod_{i=1}^{2p+1} \|f_i\|_{X^{\ell+n+4, 4n+3}(\RR^n)} ~.
\end{split}
\end{equation*}
\end{enumerate}
\end{theorem}

\end{document}